\documentclass[11pt]{amsart}
\usepackage{amsthm,amsfonts,amssymb,amsmath,oldgerm}
\numberwithin{equation}{section}
\usepackage{fullpage}
\usepackage{amsmath}
\usepackage{setspace}
\usepackage{stmaryrd}
\usepackage{mathrsfs}
\usepackage{fancyhdr}
\usepackage{graphicx}
\usepackage{enumerate}
\usepackage{bm}
\usepackage{bbm}
\usepackage{dsfont}
\usepackage{multirow}
\usepackage{psfrag}
\usepackage[font=scriptsize]{caption}
\usepackage[font=scriptsize]{subcaption}
\usepackage{listings}
\usepackage{tikz}
\usepackage{empheq}
\usepackage{cases}
\usetikzlibrary{matrix} 
\usepackage[framemethod=TikZ]{mdframed}
\mdfdefinestyle{MyFrame}{backgroundcolor=gray!20!white}
\usepackage[makeroom]{cancel}

\usepackage[top=20mm,bottom=20mm,left=16mm,right=16mm,a4paper]{geometry}

\usepackage{hyperref,bookmark}

\usepackage[normalem]{ulem}
\definecolor{violet}{rgb}{0.580,0.,0.827}

\usepackage[textwidth= \marginparwidth,textsize=tiny]{todonotes}

\hypersetup{
    colorlinks,          % color of internal links (change box color with linkbordercolor)
    citecolor=blue,        % color of links to bibliography
    filecolor=blue,      % color of file links
    urlcolor=blue,           % color of external links
    linkcolor=blue,
}
\newcommand{\ols}[1]{\mskip.5\thinmuskip\overline{\mskip-.5\thinmuskip {#1} \mskip-.5\thinmuskip}\mskip.5\thinmuskip} % overline short

\newcommand\dD{\mathrm{d}}

\def\eps{\varepsilon }

\newcommand{\J}{{\mathcal J}}

\newcommand{\f}{\frac}

\newcommand{\beq}{\begin{equation}}
\newcommand{\eeq}{\end{equation}}
\newcommand{\beqa}{\begin{eqnarray}}
\newcommand{\eeqa}{\end{eqnarray}}
\newcommand\br{\begin{remark}}
\newcommand\er{\end{remark}}
\newcommand\bp{\begin{pmatrix}}
\newcommand\ep{\end{pmatrix}}
\newcommand{\be}{\begin{equation}}
\newcommand{\ee}{\end{equation}}
\newcommand\ba{\begin{equation}\begin{aligned}}
\newcommand\ea{\end{aligned}\end{equation}}
\newcommand\ds{\displaystyle}

\newcommand{\beg}{\begin{example}}
\newcommand{\eeg}{\end{exaplem}}
\newcommand{\bpr}{\begin{proposition}}
\newcommand{\epr}{\end{proposition}}
\newcommand{\bt}{\begin{theorem}}
\newcommand{\et}{\end{theorem}}
\newcommand{\bc}{\begin{corollary}}
\newcommand{\ec}{\end{corollary}}
\newcommand{\bl}{\begin{lemma}}
\newcommand{\el}{\end{lemma}}
\newcommand{\bd}{\begin{definition}}
\newcommand{\ed}{\end{definition}}
\newcommand{\brs}{\begin{remarks}}
\newcommand{\ers}{\end{remarks}}

\newtheorem{theorem}{Theorem}[section]
\newtheorem{proposition}[theorem]{Proposition}
\newtheorem{corollary}[theorem]{Corollary}
\newtheorem{lemma}[theorem]{Lemma}
\newtheorem{remark}[theorem]{Remark}
\newtheorem{definition}[theorem]{Definition}

\newtheorem{example}[theorem]{Example}

\newcommand{\N}{{\mathbb N}}

\newcommand{\R}{{\mathbb R}}
\newcommand{\T}{{\mathbb T}}

\newcommand\bx{{\bm x}}

\newcommand\bv{{\bm v}}
\newcommand\Div{{\rm div}}

\newcommand\bE{{\mathbf E}}

\newcommand\cA{{\mathcal A}}

\newcommand\cD{{\mathcal D}}
\newcommand\cE{{\mathcal E}}

\newcommand\cH{{\mathcal H}}
\newcommand\cI{{\mathcal I}}
\newcommand\cJ{{\mathcal J}}

\newcommand\cL{{\mathcal L}}
\newcommand\cM{{\mathcal M}}

\newcommand\cR{{\mathcal R}}

\usetikzlibrary{fit}
\usetikzlibrary{arrows}
\numberwithin{equation}{section}
\numberwithin{figure}{section}
\hypersetup{linkbordercolor=red}

\title[A structure and asymptotic preserving scheme for the
Vlasov-Poisson-Fokker-Planck model]{A structure and asymptotic preserving scheme for the Vlasov-Poisson-Fokker-Planck model} 

\keywords{Hermite spectral method; Vlasov-Poisson-Fokker-Planck; Hypocoercive estimates}

\subjclass[2010]{
  Primary:
  82C40,          % Time-dependent statistical mechanics; Kinetic theory of gases
  Secondary:
  65N08,          % Numerical analysis; Finite volume methods
  65N35           % Numerical analysis; Spectral, collocation and related methods 
}

\begin{document}

\maketitle

\centerline{\scshape Alain Blaustein\footnote{akb7016@psu.edu, Pennsylvania State University, Department of Mathematics and Huck Institutes, State College, PA.} and Francis Filbet\footnote{francis.filbet@math.univ-toulouse.fr, Institut de Math\'ematiques de Toulouse, Universit\'e de Toulouse, France.}}
\medskip

\bigskip

\begin{abstract} 
We propose a numerical method for the
Vlasov-Poisson-Fokker-Planck model written as an hyperbolic system thanks to a spectral decomposition in the basis of Hermite functions with respect to the velocity variable and a structure preserving finite
volume scheme for the space variable. On the one hand, we show that this scheme naturally preserves both
stationary solutions and linearized free-energy estimate. On the other hand, we adapt previous arguments based on
hypocoercivity methods to get quantitative estimates ensuring the exponential relaxation 
to equilibrium of the discrete solution for the linearized
Vlasov-Poisson-Fokker-Planck system, uniformly with respect to both scaling and discretization parameters. Finally, we perform substantial
numerical simulations for the nonlinear system to illustrate the
efficiency of this approach for a large variety of collisional regimes
(plasma echos for weakly collisional regimes and trend to equilibrium
for collisional plasmas) and to highlight its robustness (unconditional stability, asymptotic preserving properties).
\end{abstract}

\tableofcontents

\newpage

\section{Introduction}
\label{sec:1}
\setcounter{equation}{0}
\setcounter{figure}{0}
\setcounter{table}{0}
 
The
Vlasov-Poisson-Fokker-Planck system provides a kinetic description of a gas constituted of
  charged particles, let us say electrons and heavy positive ions,   interacting through a mean electrostatic field:
\begin{equation}
    \label{eq:vpfp3D}
  \left\{
    \begin{array}{l}
\ds\frac{\partial f}{\partial t}\,+\,\bv\cdot\nabla_\bx f
      \,+\, \frac{q}{m}\,\bE\cdot\nabla_\bv f \,=\,
      \frac{1}{\tau}\,\Div_\bv\left(\bv f \,+\, T_0 \,\nabla_\bv f \right)\,,
      \\[1.1em]
      \ds\bE = -\nabla_\bx \Phi\;\;;\;\;-\eps_0\Delta_{\bx}\Phi \,=\, q\,(n- n_i)\;\;;\;\; n =\int_{\R^3} f \dD \bv \,.
      \end{array}\right.
  \end{equation}
  In \eqref{eq:vpfp3D}, $f(t,\bx,\bv)$ is the distribution of electrons over the phase space $(\bx,\bv)\in\T^3\times\R^3$ at time $t\geq0$. Field interactions are taken into account thanks to a coupling between kinetic and Poisson equations (first and second line of \eqref{eq:vpfp3D} respectively). The coupling displays several constants: vacuum permittivity $\eps_0$, elementary charge $q$, mass $m$ of an electron as well as macroscopic densities $n(t,\bx)$ of electrons and $n_i(\bx)$ of ions. Thermodynamic effects are taken into account thanks to a Fokker-Planck operator, on the right-hand side of the kinetic equation, where appears the
  relaxation time $\tau>0$ of electrons due to their collisions with the ionic surrounding bath described by a spatially homogeneous  temperature $T_0>0$.

  To investigate the physical behavior of the solution to  system
  \eqref{eq:vpfp3D},  two important quantities will be considered \cite{FN:22}.  The
  first one is $\eps$ the square root of the ratio between the mass
  of electrons and positive heavy ions, given by
  $$
  \eps \,:=\, \sqrt{\frac{m_e}{m_i}}\,\ll\, 1,
  $$
whereas the second one is $\tau(\eps)>0$, the ratio between the elapsed time
between two collisions of electrons and the observable time. We focus on a situation where the  one dimensional Vlasov-Poisson-Fokker-Planck system can be reformulated using
these parameters as
\begin{equation}
	\label{vpfp:0}
	\left\{
	\begin{array}{l}
		\ds\eps\,\partial_t f \,+\, v\,\partial_x f
		\,-\,\partial_{x}\phi \,\partial_v f \,=\,\frac{\eps}{\tau(\eps)}
		\partial_v \left( v \, f \,+\, T_0\,\partial_v f \right)\,,
		\\[1.1em]
		\ds-\partial_{x}^2\phi \,=\, \rho-\rho_i\,,\quad \rho(t,x)\,=\,\int_{\R}\,f(t,x,v)\,\dD v\,.
	\end{array}\right.
\end{equation}
This system is completed with the following condition, which ensures uniqueness of the electrical potential $\phi$
\begin{equation}
	\label{vpfp:1}
	\int_\T \phi(t,x) \,\dD x\,=\; 0\,.
\end{equation}
Since we focus on the situation where the electron to ion mass ratio
is very small, it allows to describe ions as a steady thermal bath. More precisely, ions are
supposed to be fixed with a prescribed temperature $T_0>0$ and a density
$\rho_i>0$, which is an integrable function over
$\T$. Furthermore, the following quasi-neutrality assumption is satisfied for all time $t\geq0$
$$
\int_{\T} \rho(t,x)\,\dD x \,=\,  \int_{\T} \rho_i(x)\,\dD x 
$$
as soon as this condition is initially verified.
As already mentioned, the other scaling parameter $\tau(\eps)>0$ stands for the ratio between the time which
separates two collisions of an electron with the ionic background and the timescale of observation. In this work, we suppose
$$
\tau(\eps)=\tau_0\,\eps.
$$
Therefore, as $\eps$ goes to zero, we expect that the couple  $(f,\phi)$,
solution to \eqref{vpfp:0}-\eqref{vpfp:1}, converges to $(f_\infty,\phi_\infty)$ given by
\begin{equation*}
	\left\{
	\begin{array}{l}
		\ds v\,\partial_x f_{\infty}
		\,-\,\partial_{x}\phi_{\infty} \,\partial_v f_{\infty} \,=\,\frac{1}{\tau_0}
		\partial_v \left( v \, f_{\infty} \,+\, T_0\,\partial_v f_{\infty} \right)\,,
		\\[1.1em]
		\ds-\partial_{x}^2\phi_{\infty} \,=\, \rho_{\infty}-\rho_i\,,\quad \rho_{\infty}(x)\,=\,\int_{\R}\,f_{\infty}(x,v)\,\dD v\,.
	\end{array}\right.
\end{equation*}
Actually, the equilibrium state $f_{\infty}$ is uniquely determined as follows
\[
f_{\infty}(x,v)\,=\,\rho_{\infty}(x)\,\cM(v)\,,
\]
where $\cM$ denotes the Maxwellian with temperature $T_0$
\begin{equation}
 \label{M:T0}
\cM(v)\,=\, \frac{1}{\sqrt{2\pi \,T_0}} \, \exp\left( -\frac{|v|^2}{2\,T_0}\right)\,,
\end{equation}
and where $\rho_{\infty}$ solves
\be
\label{def:rho_inf}
\left\{
\begin{array}{l}
	\ds\rho_\infty \,=\, \, c_\infty\,\exp\left(-\frac{\phi_{\infty}}{T_0} \right)\,,
	\\[1.1em]
	\ds-\partial_{x}^2\phi_{\infty} \,=\, \rho_{\infty}-\rho_i\,,
\end{array}\right.
\ee
completed with the following condition on the constant $c_\infty>0$
\[
	c_\infty\,\int_{\T} \exp\left(-\frac{\phi_{\infty}}{T_0} \right)\dD x \,=\,\int_{\T} \rho_i \,\dD x \,.
\]
Therefore, the electrons' temperature relaxes to
the background temperature whereas their spatial distribution converges to a Maxwell-Boltzmann equilibrium.\\
Let us point out that there is an attractive version of the
Vlasov-Poisson-Fokker-Planck system \eqref{vpfp:0}-\eqref{vpfp:1}
which is also widely used in stellar physics. For that case the
repulsive electrostatic force is replaced by the attractive
gravitational force, responsible for a change of sign in the Poisson
equation. In this paper, we only consider the repulsive case.
%%%%%%%%%%%%%%%%%%%%%%%%%%%%%%%%%%%%%%%%%%%%%%%%%%%%%%%%%%%%%%%%%
%
%   FF BIBLIO
%
%%%%%%%%%%%%%%%%%%%%%%%%%%%%%%%%%%%%%%%%%%%%%%%%%%%%%%%%%%%%%%%%%%
The numerical resolution of the system
\eqref{vpfp:0}-\eqref{vpfp:1} shares the same difficulties as most
of  kinetic equations: high dimensionality,
presence of various scales, etc... Several numerical methods have been
developed for \eqref{vpfp:0}-\eqref{vpfp:1}, we mention for
instance \cite{VH:98,Ozi:05, cruz, DD, car:21, Tao2015}. These numerical schemes are
either deterministic or stochastic, with an effort to capture some physical
phenomena associated to weakly collisional plasmas such as Landau
damping or two-stream instability, occurring for short time range, before being canceled by
collisions.  More recently, dynamical low-rank algorithms have been
proposed \cite{Lukas2021, Hu2022}, they  decouple the dimensions of the phase space allowing to reduce the computational cost.

Here, we want to design a numerical scheme  preserving the large-time
behavior of solutions to the Vlasov-Poisson-Fokker-Planck equation or related
nonlinear kinetic models. More precisely, we want to describe
accurately the transient regime in which transport effects dominate and where the distribution function exhibits
filamentation, but also large time scales
in which collisions take over, forcing the relaxation of the distribution to thermal
equilibrium, given here by a Maxwell-Boltzmann distribution. Let us remind that
when we neglect both the coupling with the
Poisson equation and the free transport term,  a fully discrete finite difference
scheme for the homogeneous Fokker-Planck equation has been proposed in
the pioneering work of Chang and Cooper \cite{ChangCooper:1970}. This
scheme has nice properties since it preserves the stationary solution and the entropy decay of the
numerical solution. Later, finite volume schemes preserving the
exponential trend to equilibrium have been studied for nonlinear
convection-diffusion equations (see for example
\cite{Bessemoulin-Chatard2012,Burger2010,Chainais-Hillairet2007}
and more recently \cite{Georgoulis,
  Dujardin,MMT} in the frame of hypocoercive methods). Let
us also mention that in \cite{Pareschi2017}, authors investigate the
question of describing correctly the equilibrium state of nonlinear
diffusion and kinetic models for high order schemes.
\\

More recently, in our previous
work \cite{BF_09_22}, we have proposed and carefully studied a numerical scheme based
on Hermite polynomials in the velocity space for the
linear Vlasov-Fokker-Planck equation when the electric
field $E_\infty=-\partial_x\phi_\infty$ is prescribed. This approach preserves the
equilibrium and thanks to a discrete hypocoercive method, it is shown
that numeric solutions relax exponentially fast to thermodynamic equilibrium. When collisions are neglected, Hermite decomposition has also been successfully applied  to approximate the
Vlasov-Poisson system on a finite time interval \cite{Filbet2020, ref:5, bessemoulin2022cv}. Therefore, our goal is to provide here an
efficient approximation of the nonlinear Vlasov-Poisson-Fokker-Planck
system \eqref{vpfp:0}-\eqref{vpfp:1} able to describe
accurately  both weakly and strongly collisional regimes
and to preserve the long time behavior of the solution. To illustrate
the feature of our numerical scheme, we will analyze the asymptotic behavior of the
solution given by the discrete linearized system and then propose
various numerical experiments in various regimes.

 In order to design a well-balanced approximation for
 \eqref{vpfp:0}-\eqref{vpfp:1}, we consider an equivalent
 reformulation where we  define a potential $\psi
 = \phi-\phi_\infty$ and  replace the  electrical potential $\phi$ in
 the equation, it yields that  $\rho_i$ is now replaced by the quantities at
 equilibrium as follows, 
\begin{equation}
	\label{vpfp:new0}
	\left\{
	\begin{array}{l}
		\ds\eps\,\partial_t f \,+\, v\,\partial_x f
		\,-\,\partial_{x}\phi_{\infty} \,\partial_v f 
		\,-\,\partial_{x}\psi \,\partial_v f 
		\,=\,\frac{1}{\tau_0}
		\partial_v \left( v \, f \,+\, T_0\,\partial_v f \right)\,,
		\\[1.1em]
		\ds-\partial_{x}^2\psi \,=\, \rho-\rho_{\infty}\,,\quad \rho(t,x)\,=\,\int_{\R}\,f(t,x,v)\,\dD v\,,
	\end{array}\right.
  \end{equation}
coupled with the condition on $\psi$,
 \begin{equation}
	\label{vpfp:new1}
\int_\T \psi(t,x) \,\dD x\,=\, 0\,.
\end{equation}
The equilibrium to \eqref{vpfp:new0}-\eqref{vpfp:new1} is now characterized by
$(f_\infty,\psi_\infty)$ where $\psi_\infty\equiv 0$.

%%%%%%%%%%%%%%%%%%%%%%%%%%%%
% FF : Bibliographie trend to equilibrium
%%%%%%%%%%%%%%%%%%%%%%%%%%%%

The key-estimate to prove the trend to equilibrium of solutions to \eqref{vpfp:new0}-\eqref{vpfp:new1} is given by
\begin{equation*}
  \frac{\dD }{\dD t} \cH(f,f_\infty)  \,\,=\,\, - \,\frac{1}{\eps\,\tau_0}
  \,\cI(f, f_\infty), 
\end{equation*}
where $ \cH(f,f_\infty)$ denotes the free energy
\begin{equation*}
  %	\label{estim:ener}
\cH(f,f_\infty) \,:=\, \int_{\T\times\R} \ln{\left(\frac{f}{f_{\infty}}\right)}\,f\,\dD x\,\dD v\,+\,\frac{1}{2\,T_0}
\left\|\,\partial_x\psi\,\right\|_{L^2\left(\T\right)}^2\,,
\end{equation*}
and $\cI(f,f_\infty)$ is the entropy dissipation 
\begin{equation*}
\cI(f,f_\infty)\,:=\, 4\,T_0\,
\int_{\T\times\R} \left|
\partial_v
\sqrt{\frac{f}{f_{\infty}}}\,
\right|^2
f_{\infty}\, \dD x\,\dD v\,.
\end{equation*}
The free energy estimate is said to be "hypocoercive" since we have
\[
\cI(f(t),f_\infty)\,=\,0\quad\iff \quad f(t,x,v)\,=\,\rho(t,x)\,\cM(v)\,,\quad \forall(x,v)\in\T\times\R\,,
\]
meaning that the entropy dissipation $\cI$ controls at most the distance between $f$ and its associated \textbf{local} equilibrium $\rho\,\cM$, giving no straightforward information on the long time behavior of $\rho$. \\
In the linear case, corresponding to $\psi\equiv 0$ in \eqref{vpfp:new0}, L. Desvillettes and
C. Villani have proposed a general method, based on the latter
relative entropy estimate and logarithmic Sobolev
inequalities, to overcome this degeneracy in the
position variable \cite{DV:2001}. This
approach has been widely explored in the last decade \cite{Villani:AMS, Dolbeault2015}.
However, when the Vlasov-Fokker-Planck equation is coupled with the
Poisson equation for the electrical potential, a different
functional framework, based on weighted $L^2$ spaces, is applied motivated by the following estimate when $f$ is near
equilibrium \cite{Herau_Thomann16, Herda_Rodrigues}. Indeed, plugging the
following formal expansion into the free energy
\[
f\,\ln{\left(\frac{f}{f_{\infty}}\right)}
\,\underset{f \rightarrow f_{\infty}}{\sim}\,
f-f_{\infty}
\,+\,
\frac{\left|f-f_{\infty}\right|^2}{2\,f_{\infty}}
\]
and using that mass is conserved for solutions to
\eqref{vpfp:new0}-\eqref{vpfp:new1}, we define a new functional, named
the linearized free energy, as
\begin{equation}
  \label{continuous:energy}
\cE(t)\,=\,
\left\|\,f(t)-f_{\infty}\,\right\|^2_{L^2\left(f_{\infty}^{-1}\right)}\,+\,
\frac{1}{T_0}
\left\|\,\partial_x\psi(t)\,\right\|^2_{L^2\left(\T\right)}\,.
\end{equation}
Unfortunately, this functional is not dissipated for the solution to
the nonlinear system \eqref{vpfp:new0}-\eqref{vpfp:new1}, but only
for its linearized version given by 
\begin{equation}
\label{vpfp:lin0}
\left\{
\begin{array}{l}
	\ds\eps\,\partial_t f \,+\, v\,\partial_x f
	\,-\,\partial_{x}\phi_{\infty} \,\partial_v f 
	\,-\,\partial_{x}\psi \,\partial_v f_\infty 
	\,=\,\frac{1}{\tau_0}
	\partial_v \left( v \, f \,+\, T_0\,\partial_v f \right)\,,
	\\[1.1em]
	\ds-\partial_{x}^2\psi \,=\, \rho-\rho_{\infty}\,,\quad \rho(t,x)\,=\,\int_{\R}f(t,x,v)\,\dD v\,,
\end{array}\right.
\end{equation}
coupled with the condition on $\psi$,
 \begin{equation}
	\label{vpfp:lin1}
\int_\T \psi(t,x) \,\dD x\,=\, 0\,.
\end{equation}
This yields for the solution $(f,\psi)$ to \eqref{vpfp:lin0}-\eqref{vpfp:lin1} (see Proposition \ref{prop:2.1} for a complete proof)
\begin{equation}\label{key:estimate}
\frac{1}{2}\,\frac{\dD}{\dD t}\,
\cE(t)
\,=\,
-\,\frac{T_0}{\eps\,\tau_0}
\int_{\T\times\R} \left|
\partial_v
\left(\frac{f(t)}{f_{\infty}}\right)\,
\right|^2
f_{\infty}\, \dD x\,\dD v\,.
\end{equation}

The purpose of this paper is to design a numerical scheme for the
nonlinear Vlasov-Poisson-Fokker-Planck system
\eqref{vpfp:new0}-\eqref{vpfp:new1} for which such estimate occurs on its linearized version \eqref{vpfp:lin0}-\eqref{vpfp:lin1}. To this aim, we propose a simple time
splitting scheme, where the first stage consists in solving the linearized
system \eqref{vpfp:lin0}-\eqref{vpfp:lin1}, whereas the second
stage solves the remaining quadratic part of \eqref{vpfp:new0}-\eqref{vpfp:new1}, that is
\begin{equation*}
			\eps\,\partial_t f  \,-\,\partial_{x}\psi \,\partial_v (f - f_\infty) \,=\,0\,,
\end{equation*}
for which $\psi$ is unchanged.

This approach has several advantages from the computational and
stability point of view. Indeed, both steps will be fully implicit in time allowing to use a large time step, uniformly with respect to the parameter $\eps$. Moreover, the linearized equation is autonomous, hence it requires to solve the \textbf{same} linear system at each time step, which can be done using a direct solver. Furthermore, solving the time dependent implicit second step is in fact negligible in terms of computational costs since the associated system is trivially invertible due to the Hermite discretization.
Finally, the numerical approximation of the linearized system allows to capture a consistent asymptotic profile when $\eps\rightarrow 0$ and it also preserves the free energy estimate as in \cite{BF_09_22}, which treats the case of a Vlasov-Fokker-Planck equation without a coupling with Poisson.

In Section \ref{sec:2} we propose a numerical discretization of the
full model \eqref{vpfp:new0}-\eqref{vpfp:new1} based on Hermite's
decomposition in the velocity space and finite volume scheme for the
space discretization. Then, in Section \ref{sec:3}, we prove
quantitative properties on its linearized version
\eqref{vpfp:lin0}-\eqref{vpfp:lin1}. More precisely, we first prove
a discrete version of the free energy estimate \eqref{key:estimate}
and then, using discrete hypocoercive methods, we prove the exponential
trend to equilibrium with rate $1/\eps$, this uniformly with respect to discretization parameters. This result constitutes a theoretical proof of the asymptotic-preserving properties of the method at the linearized level. Finally, in Section
\ref{sec:4}, we carry out numerical experiments which illustrate the
robustness of our scheme in a wide variety of situations ranging from
near-to-collisionless regime $1\ll \tau_0$ to the stiff limit
$\eps\sim 0$ and including inhomogeneous ionic background.  In particular, we highlight the asymptotic preserving properties of the method. Furthermore, we observe formation of nonlinear echoes and study their suppression in weakly collisional settings as well as simultaneous vortex/filamentation formation for inhomogeneous ionic background. These phenomena have drawn intense mathematical interest in the kinetic community over the past decade \cite{Bedrossian17,Grenier_Toan_Rodnianski22,Chaturvedi_Luk_Nguyen23}. With this work, we aim at taking part in these efforts by proposing numerical methods capable to capture these phenomena efficiently.

\section{Numerical scheme}
\label{sec:2}
\setcounter{equation}{0}
\setcounter{figure}{0}
\setcounter{table}{0} 
In this section, we will introduce our numerical method. We follow the lines of our previous work for the linear
Vlasov-Fokker-Planck equation \cite{BF_09_22} and then propose a
discretization of the Poisson equation allowing to preserve the
energy estimate for the linearized problem
\eqref{vpfp:lin0}-\eqref{vpfp:lin1}. Finally, the nonlinear part is
discretized using an implicit scheme.

\subsection{Hermite's decomposition for the velocity variable}
Let us first focus on the discretization of the velocity variable. It consists in performing a spectral decomposition of the distribution $f$ into its Hermite modes $\left(\Psi_k\right)_{k\in\N}$ defined as
\[
\Psi_{k}(v)\,=\, H_{k}\left(\frac{v}{\sqrt{T_0}}\right)\,\cM(v)\,,
\]
and which constitute an orthonormal system for the inverse Gaussian weight since it holds
\[
\int_{\R}\,
\Psi_{k}(v)\,\Psi_{l}(v)\,\cM^{-1}(v)\dD v
\,=\,
\delta_{k,l}\,,
\]
where $\cM$ is the Maxwellian corresponding to the stationary state of
the Fokker-Planck operator \eqref{M:T0} and $\delta_{k,l}$ the Kronecker symbol ($\delta_{k,l}=1$ when $k=l$ and
$\delta_{k,l}=0$ otherwise). In the latter definition, $
\ds
\left(
H_{k}
\right)_{k\in\N} 
$ stands for the family of Hermite polynomials defined recursively as follows
$H_{-1}=0$, $H_{0}=1$ and
\[
\xi\,H_{k}(\xi)\,=\,
\sqrt{k}\,H_{k-1}(\xi)
        \,+\,
        \sqrt{k+1}\,H_{k+1}(\xi)
\,,\quad\forall\, k\,\geq\,0\,.
\]
Let us also point out that Hermite's polynomials verify the following relation
\[
H_k'(\xi)\,=\,\sqrt{k}\,H_{k-1}(\xi)\,,\quad\forall\, k\,\geq\,0\,.
\]
The Hermite system arises naturally in our context since it offers a simple discrete reformulation of the $L^2\left(f_{\infty}^{-1}\right)$-norm which appears in the key estimate \eqref{key:estimate}, indeed it holds
\[
\left\|f(t)\right\|^2_{L^2\left(f_{\infty}^{-1}\right)}
\,=\,
\sum_{k\in\N}
\left\|C_k(t)\right\|^2_{L^2\left(\rho_{\infty}^{-1}\right)}\,,
\]
where $
C\,=\,
\left(
C_{k}
\right)_{k\in\N}
$ stand for the Hermite components of $f$
\begin{equation*}
	f\left(t,x,v\right)
	\,=\,
	\sum_{k\in\N}\,
	C_{k}
	\left(t,x
	\right)\,\Psi_{k}(v)\,.
\end{equation*}
As one can see in the latter relation, each term of the sequence $
C\,=\,
\left(
C_{k}
\right)_{k\in\N}
$ naturally belongs to the weighted space $L^2\left(\rho_{\infty}^{-1}\right)$. From the numerical point of view, working in weighted spaces induces difficulties when it comes to integro/differential manipulation such as integration by part. This is the reason why rather than discretizing coefficients $
C\,=\,
\left(
C_{k}
\right)_{k\in\N}
$, we consider their re-normalized versions $
D\,=\,
\left(
D_{k}
\right)_{k\in\N}
$ defined as 
\begin{equation}
  \label{f:decomp}
	f\left(t,x,v\right)
	\,=\,\sqrt\rho_\infty(x)\,
	\sum_{k\in\N}\,
	D_{k}
	\left(t,x
	\right)\,\Psi_{k}(v)\,.
\end{equation}
According to latter considerations, renormalized Hermite coefficients $
D\,=\,
\left(
D_{k}
\right)_{k\in\N}
$ verify
\[
\left\|f(t)\right\|^2_{L^2\left(f_{\infty}^{-1}\right)}
\,=\,
\sum_{k\in\N}
\left\|D_k(t)\right\|^2_{L^2\left(\T\right)}\,.
\]
To sum up, renormalized Hermite coefficients play a fundamental role in our analysis for two reasons: they offer a discrete reformulation of the key quantity $\cE(t)$ given by \eqref{continuous:energy} and they belong to the \textbf{unweighted} $L^2$-Lebesgue space over $\T$. Moreover, there is another benefit coming out of this choice: thanks to the properties of Hermite polynomials, one can see that Hermite functions diagonalize the Fokker-Planck operator since it holds
\[
\partial_{v}\left[\,
v\,\Psi_{k}
\,+\,
T_0\,\partial_{v}\,\Psi_{k}\,
\right]
\,=\,
\,-\,k
\,\Psi_{k}\,.
\]
Therefore, following \cite{BF_09_22}, we substitute the decomposition
\eqref{f:decomp} in the first line of \eqref{vpfp:new0} : using the identities $E_\infty =
-\partial_x\phi_\infty$ and $\rho_\infty \, E_{\infty} \,=\, T_0\, \partial_x\rho_\infty $, we get
that $D\,=\,(D_{k})_{k \in \N}$ satisfies the following system
  \begin{equation}
  \label{Hermite:D}
  \left\{
    \begin{array}{l}
  \ds\eps\,\partial_t D_{k} \,+\,
       \sqrt{k}\,
      \cA\,D_{k-1}\,- \, \sqrt{k +1}\,
        \cA^\star D_{k+1}\,+\,\sqrt{\frac{k}{T_0}}\, \partial_x\psi\,D_{k-1}
          \,  =\,  - \,  \frac{k}{\tau_0}\,  D_{k}\,, \qquad \forall\,k\in\N\,,
        \\[1.2em]
         \ds D_{k}(t=0) =  D^{0,\eps}_{k}\,, \qquad \forall\,k\in\N\,,
\end{array}\right.
\end{equation}
{where we set $D_{-1}=0$}. Due to our unweighted $L^2$-framework, the latter formulation enjoys a nice duality structure since $\cA$ and $\cA^{\star}$ are adjoint operators in $L^2\left(\T\right)$. More precisely, $\cA$ and $\cA^{\star}$ are given by
\begin{equation*}
%\label{AA}
\left\{
\begin{array}{l}
\ds  \cA \,u  \,=\, +\sqrt{T_0}\,\partial_{x}  u \,-\,
\frac{E_{\infty}}{2\sqrt{T_0}}\, u\,, \\[1.1em]
\ds\cA^\star \,u  \,=\, -\sqrt{T_0}\,\partial_{x}  u \,-\,
\frac{E_{\infty}}{2\sqrt{T_0}}\, u\,.
\end{array}\right.
\end{equation*}
Notice that both operators $\cA$ and $\cA^\star$ may also be rewritten as follows
\begin{equation*}
	%\label{AA}
	\left\{
	\begin{array}{l}
		\ds  \cA \,u  \,=\, +\sqrt{T_0\,\rho_{\infty}}\,\partial_x
		\left(
		\frac{u}{\sqrt{\rho}_{\infty}}\right)
		\,, \\[1.1em]
		\ds  \cA^{\star} \,u  \,=\, -\sqrt{\frac{T_0}{\rho_{\infty}}}\,\partial_x
		\left(
		\sqrt{\rho}_{\infty}\,u\right)
		\,.
	\end{array}\right.
\end{equation*}
To conclude, we denote by $D_\infty$ the Hermite decomposition of the
equilibrium $f_{\infty}$. It is determined for all {$k\in\N$} by
\begin{equation}
  \label{Dinf}
D_{\infty,k} \,=\,
\left\{
  \begin{array}{l}
    \sqrt\rho_\infty, \,\, {\rm if } \,\, k=0\,,
    \\[0.9em]
    0, \, \, {\rm else\,.}
    \end{array}
  \right.
\end{equation}
\subsection{Poisson equation formulated in the Hermite framework}\label{sec:form:poisson:hermite}
To compute the electrical potential $\psi$, we will reformulate the Poisson equation in such a way
that the free energy estimate \eqref{key:estimate} for the linearized system
\eqref{vpfp:lin0}-\eqref{vpfp:lin1} is satisfied. The main idea to
preserve such estimate is to construct a scheme for which it is possible to perform at the discrete level analogous calculations as the ones needed at the continuous level to derive the free
energy estimate \eqref{key:estimate}.  To this aim, we introduce a
modified potential $\omega$ given by
$$
\omega\,=\,\frac{\sqrt{\rho}_{\infty}}{T_0}\,\psi\,.
$$
Using the definition of the operator $\cA$, the electric field
$E=-\partial_x \psi $ is now given by
$$
\sqrt\rho_\infty  \,E \,\,=\,\, -\sqrt\rho_\infty \,\partial_x \psi  \,=\, -\sqrt{T_0} \,\cA\, \omega
$$
and the modified potential $\omega$ solves the modified Poisson equation
\begin{equation}
  \label{Hermite:v}
\left(\cA^\star\,\rho_{\infty}^{-1}\,\cA\right) \omega\, =\, D_0-\sqrt{\rho}_{\infty}\,.
\end{equation}
This new formulation will allow us to easily construct a numerical scheme for
the Poisson equation preserving the key energy estimate \eqref{key:estimate}, where in this framework, the linearized system \eqref{vpfp:lin0}-\eqref{vpfp:lin1} rewrites 
\begin{equation}
	\label{Hermite:lin}
	\left\{
	\begin{array}{l}
		\ds\partial_t D_{k} \,+\,
		\frac{1}{\eps}\,\left( \sqrt{k}\,
		\cA\,D_{k-1}\,- \, \sqrt{k +1}\,
		\cA^\star D_{k+1}\,+\, \cA\,  \omega \,\delta_{k,1} 
		\right)  \,  =\,  - \,  \frac{k}{\eps\,\tau_0}\,  D_{k}\,, \qquad{\forall\,k\in\N}\,,
		\\[1.2em]
		\ds \left(\cA^\star\,\rho_{\infty}^{-1}\,\cA\right) \,\omega\, =\, D_0-\sqrt{\rho}_{\infty}\,,
		\\[1.2em]
		\ds D_{k}(t=0) =  D^{0,\eps}_{k}\,, \qquad{\forall\,k\in\N}\,,
	\end{array}\right.
\end{equation}
completed with the condition 
\begin{equation}
\int_{\T} \frac{\omega}{\sqrt\rho_{\infty}} \,\,\dD x \,=\,0\,,
\end{equation}
and where the linearized free energy $\cE$ reads
      \begin{equation}
          \label{E:D}
	\cE(t)
	\,=\,
	\frac{1}{2}
	\left(
	\left\|
	D(t)-D_{\infty}
	\right\|^2_{L^2}
	\,+\,
	\left\|
	\frac{\cA\,\omega}{\sqrt{\rho}_{\infty}}
	\right\|^2_{L^2\left(\T\right)}
	\right)\,,
      \end{equation}
      where $\|\cdot\|_{L^2}$ stands for the overall $L^2$-norm \textbf{with no weight}
\[
\|D\|^2_{L^2}
\,=\,
\sum_{
	k\in\N
}
\| D_{k}\|_{L^2\left(\T\right)}^2\,.
\]
From this reformulated equation,  we first prove the free energy
estimate on the linearized system  \eqref{vpfp:lin0}-\eqref{vpfp:lin1}.
\begin{proposition}
\label{prop:2.1}
Consider a {formal solution} $(D_k)_{k\in\N}$ to 
\eqref{Hermite:lin} such that for all $t>0$,
{$$
\sum_{k\geq 0} k\, \| D_k(t)\|_{L^2}^2 < \infty.
$$}
The following free energy estimate holds for all $t\geq0$
\be
\label{estim:L2}
\frac{\dD }{\dD t}\cE(t) \,+\,
\frac{1}{\eps\, \tau_0}\,\sum_{k\in\N^\star} k\,\left\|D_k(t)\right\|^2_{L^2\left(\T\right)}
\,=\, 0\,.
\ee
\end{proposition}
\begin{proof}
We multiply equation \eqref{Hermite:lin} by $D_k-D_{\infty,k}$, sum
over all $k \in \N$, integrate in $x\in \T$ and then {rearrange all the terms}, this yields
$$
\frac{1}{2} \frac{\dD }{\dD t} \| D - D_\infty\|_{L^2}^2  \,+\, \frac{1}{\eps} \int_{\T} \cA\,\omega \, D_1  \dD x
\,=\, - \frac{1}{\eps\,\tau_0} \sum_{k\in\N^*} k \|D_k\|_{L^2\left(\T\right)}^2.
$$
We rewrite the integral term in the latter estimate using the duality structure and the equation on $D_0$
$$
 \int_{\T} \cA\,\omega \, D_1 \dD x \,=\, \int_{\T} \omega \, \cA^\star D_1 \dD
 x  \,=\, \eps\,\int_{\T} \omega \, \partial_t (D_0 - \sqrt{\rho}_\infty)\,\dD x. 
 $$
 Using the reformulated Poisson equation in \eqref{Hermite:lin} and the duality structure again, we deduce
\[
\eps
\int_{\T} \omega\,\partial_t (D_0- \sqrt{\rho}_\infty)\,\dD x
\,=\,\eps
\int_{\T} \omega\,\partial_t\left(\cA^\star\,\rho_{\infty}^{-1}\,\cA\right) \omega \,\dD x
\,=\,\frac{\eps}{2}\,
\frac{\dD}{\dD t}
\left\|\frac{\cA\,\omega}{\sqrt{\rho}_{\infty}}\right\|^2_{L^2\left(\T\right)}
\,.
\]
It finally yields the free
 energy estimate
 $$
\frac{\dD }{\dD t}\cE(t)  \,=\,
-\frac{1}{\eps\, \tau_0}\,\sum_{k\in\N^\star}
k\,\left\|D_k(t)\right\|^2_{L^2\left(\T\right)}\,.
 $$
\end{proof}
Let us now draw two conclusions from these considerations. On the one hand, we observe that the estimate given in Proposition \ref{prop:2.1} is not sufficient to prove convergence of
the solution to the linearized system \eqref{Hermite:lin} towards the stationary state
\eqref{Dinf} because of the lack of coercivity. To bypass this difficulty, we will define a modified relative energy $\cH$ as
\begin{equation}
	\cH(t) \,=\,
	\cE(t)\,+\, \beta_0\,\left\langle
	\cA^\star D_1(t),\, u(t)\right\rangle_{L^2\left(\T\right)}\,,
	\label{eq:H0}
\end{equation}
where $\beta_0>0$ is a small free parameter and $u$ is solution to
$$
\left\{
\begin{array}{l}
	\ds \cA^\star\,\cA \,u\ =\ D_0\,-\,\sqrt{\rho}_{\infty}\,,
	\\[1.1em]
	\ds\int_{\T}u \,\sqrt\rho_{\infty}\,\dD x \,=\,0\,.
\end{array}\right.
$$
To get the convergence to the solution to the linearized system
\eqref{Hermite:lin} to the stationary state, the strategy consists in
proving that $\cH$ and $\cE$ are equivalent and that there exists a
constant  $\kappa>0$ such that
$$
\frac{\dD}{\dD t} \cH(t) \leq - \frac{\kappa}{\eps}\,\min\left(\tau_0,\tau_0^{-1}\right)\, \cH(t).
$$
We do not detail these results in the continuous case since they constitute the object of Section \ref{sec:3} in the discrete framework.\\
On the other hand, from the reformulated Poisson equation for the linearized system \eqref{vpfp:new0}-\eqref{vpfp:new1}, we are now able to write the Hermite method for the full Vlasov-Poisson-Fokker-Planck
system, which reads
\begin{equation*}
 % \label{Hermite:nlD}
  \left\{
    \begin{array}{l}
  \ds\eps\,\partial_t D_{k} \,+\,
      \sqrt{k}\,
      \cA\,D_{k-1}\,- \, \sqrt{k +1}\,
        \cA^\star D_{k+1}\,+\,\sqrt{\frac{k}{\rho_\infty}}\, \cA\,
      \omega\, D_{k-1}
      \,  =\,  - \,  \frac{k}{\tau_0}\,  D_{k}\,, \qquad {\forall\,k\in\N}\,,
         \\[1.2em]
      \ds \left(\cA^\star\,\rho_{\infty}^{-1}\,\cA\right) \,\omega\, =\, D_0-\sqrt{\rho}_{\infty}\,,
        \\[1.2em]
         \ds D_{k}(t=0) =  D^{0,\eps}_{k}\,, \qquad{\forall\,k\in\N}\,,
\end{array}\right.
\end{equation*}
completed with the condition 
\begin{equation*}
	\int_{\T} \frac{\omega}{\sqrt\rho_{\infty}} \,\,\dD x \,=\,0\,.
\end{equation*}
This latter formulation  is equivalent to the initial
  system \eqref{Hermite:D}. The Hermite spectral method now consists
  in considering a finite number of modes $N_H$. Therefore we solve for $k=0,\ldots, N_H$,
\begin{equation}
  \label{Hermite:nlD}
  \left\{
    \begin{array}{l}
  \ds\eps\,\partial_t D_{k} \,+\,
      \sqrt{k}\,
      \cA\,D_{k-1}\,- \, \sqrt{k +1}\,
        \cA^\star D_{k+1}\,+\,\sqrt{\frac{k}{\rho_\infty}}\, \cA\,
      \omega\, D_{k-1}
      \,  =\,  - \,  \frac{k}{\tau_0}\,  D_{k}\,,
         \\[1.2em]
      \ds \left(\cA^\star\,\rho_{\infty}^{-1}\,\cA\right) \,\omega\, =\, D_0-\sqrt{\rho}_{\infty}\,,
        \\[1.2em]
         \ds D_{k}(t=0) =  D^{0,\eps}_{k}\,,
\end{array}\right.
\end{equation}
where we set $D_{-1}=D_{N_H+1}=0$. It is worth to mention that Proposition \ref{prop:2.1} also holds true for \eqref{Hermite:nlD}. {Indeed, this
truncation does not modify the above computations on the free energy
estimate.}

In the following subsection, we design a well-balanced finite volume discretization of
the nonlinear system \eqref{Hermite:nlD} such that the associated
approximation of the linearized system \eqref{Hermite:lin} satisfies the estimate given in Proposition \ref{prop:2.1} (Section \ref{sec:3}). 

\subsection{Finite volume discretization for the space variable}
We now turn to the phase space and time discretization of the system
\eqref{Hermite:nlD} and propose a well-balanced time splitting
scheme. \\
To discretize the phase space domain, we fix a number of Hermite modes $N_H\in\N^*$. Then, we consider an interval $(a,b)$ of $
\mathbb{R}$ and for $N_{x}\in\N^\star$ an odd number, introduce the set
$\J=\{1,\ldots, N_x\}$ and a family of control volumes
$\left(K_{j}\right)_{j\in\J}$ such that
$K_{j}=\left]x_{j-{1}/{2}},x_{j+{1}/{2}}\right[$ with $x_{j}$ the
middle of the interval $K_j$ and 
\begin{equation*}
a=x_{{1}/{2}}<x_{1}<x_{{3}/{2}}<...<x_{j-{1}/{2}}<x_{j}<x_{j+{1}/{2}}<...<x_{N_{x}}<x_{N_{x}+{1}/{2}}=b\,.
\end{equation*}
Let us set
$$
\left\{
  \begin{array}{l}
\ds\Delta x_{j}=x_{j+{1}/{2}}-x_{j-{1}/{2}}, \,\text{ for } j \in\J\,,
\\[0.9em]
\ds\Delta x_{j+{1}/{2}} = x_{j+1}-x_{j}, \,\text{ for } 1 \leq j \leq N_{x}-1\,.
  \end{array}\right.
$$
We also introduce the parameter $h$ such that
$$
h \,=\, \max_{j \in\J} \Delta x_j\,.
$$
To discretize the time variable, we fix a time step $ \Delta t$ and we set $t^{n}=n \Delta t$ with
$n\in\N$. Our time discretization of $\R^+$ is then given by the
increasing sequence of $(t^{n})_{n \in\N}$. \\
\be
\label{discrete:step1}
\left\{
\begin{array}{l}
\ds
\eps\,\frac{D^{n+1/2}_{1} -D_{1}^{n} }{\Delta t}
\,+\,\mathcal{A}_{h}\,D^{n+1/2}_{0}
\,-\,\sqrt{2}\,\mathcal{A}_{h}^\star\,D^{n+1/2}_{2} \,+\, \cA_h\,\omega_h^{n+1/2} \,=\,
-\, \frac{1}{\tau_0} D^{n+1/2}_{1}\,,
\\[1.em]
\ds
\eps\,\frac{D^{n+1/2}_{k} -D_{k}^{n} }{\Delta t}  \,+\,
\sqrt{k}\,\mathcal{A}_{h}\,D^{n+1/2}_{k-1}
\,-\,\sqrt{k+1}\,\mathcal{A}_{h}^\star\,D^{n+1/2}_{k+1}
\,=\,-\, \frac{k}{\tau_0}\, D^{n+1/2}_{k}\,, \,{\rm for }\, k\neq 1\,,\\[1.2em]
\ds (\cA_h^\star\,\rho_{\infty}^{-1}\,\cA_h) \,\omega_h^{n+1/2}\, =\, D^{n+1/2}_0-\sqrt{\rho}_{\infty}\,,
\\[1.1em]
\ds\sum_{j\in\J}\Delta x_j \, \omega_j^{n+1/2} \,\sqrt\rho^{\,-1}_{\infty,j}\,=\,0\,,
\end{array}\right.
\ee
for $k\in\{0,\ldots,N_H\}$ and $D_k^{n+1/2}=0$ when $k>N_H$ and $k=-1$. Moreover, operator $\mathcal{A}_{h}$ (resp. $\mathcal{A}_{h}^\star$) is an
approximation of the operator $\cA$ (resp. $\cA^\star$) given by
\be
\label{eq:Ah}
\cA_h = (\cA_j)_{j \in\J} \quad{\rm and}\quad \cA_h^\star = (\cA_j^\star)_{j \in\J} 
\ee
and where for
$D=(\cD_{j})_{j\in\J}$ it holds
\be
\label{def:Ah}
\left\{
\begin{array}{l}
	\ds\cA_{j} D \,=\,  \ds\sqrt{T_0}\,\frac{\cD_{j+1} - \cD_{j-1}}{2\Delta
		x_j}  \,-\, \frac{E_{\infty,j}}{2\,\sqrt{T_0}}\, \cD_{j}\,, \quad j\in\J\,,
	\\[1.1em]
	\ds\cA_{j}^\star D \,=\, -\ds\sqrt{T_0}\,\frac{\cD_{j+1} - \cD_{j-1}}{2\Delta
		x_j}  \,-\, \frac{E_{\infty,j}}{2\,\sqrt{T_0}}\, \cD_{j}\,, \quad j \in\J\,,
\end{array}\right.
\ee
whereas the discrete electric field $E_{\infty,j}$ is given by
\be
\label{def:Ei}
E_{\infty,j} \,=\, -\frac{\phi_{\infty,j+1}-\phi_{\infty,j-1}}{2\Delta x_j} \,=\, \frac{2\,T_0}{\sqrt\rho_{\infty,j}}\,\frac{\sqrt\rho_{\infty,j+1}- \sqrt\rho_{\infty,j-1}}{2\,\Delta x_j}\,,
\ee
where $\rho_{\infty,j}$ is an approximation of the stationary density $\rho_{\infty}$ on the cell
$K_j$. This latter formula is consistent with the definition of
$\sqrt\rho_{\infty}= \exp\left({-\phi_{\infty}/(2T_0)}\right)$ and the fact that
$$
\frac{1}{2\,T_0}\,\partial_x\phi_{\infty} \,=\, -\frac{1}{\sqrt\rho_{\infty}} \, \partial_x\sqrt\rho_{\infty}\,,
$$
furthermore, our discretization of the field $E_{\infty}$ allows to preserve the equilibrium since it ensures
\begin{equation}\label{preserves:equi:h}
\cA_j \sqrt{\rho}_{\infty}\,=\,
\sqrt{T_0}\,\frac{\sqrt{\rho}_{\infty,j+1} - \sqrt{\rho}_{\infty,j-1}}{2\Delta
	x_j}  \,-\, \frac{E_{\infty,j}}{2\,\sqrt{T_0}}\, \sqrt{\rho}_{\infty,j}
	\,=\,
0\,,\quad\forall j\in\cJ\,.
\end{equation}
This first step requires the numerical resolution of a linear system
which does not depend on the time index $n$. Hence a direct solver based on
$LU$ factorization is applied to get the solution
$\left(D^{n+1/2},\omega_h^{n+1/2}\right)$.\\

On the other hand, we solve the nonlinear part using
again a fully implicit Euler scheme
\be
	\label{discrete:step2}
	\left\{
	\begin{array}{lll}
		\ds
		D^{n+1}_{0}\,=\,D^{n+1/2}_{0}\,,
		\\[1.em]
		\ds
		\omega^{n+1}_{h}\,=\,\omega^{n+1/2}_{h}\,,
		\\[1.em]
		\ds
		\eps\,\frac{D^{n+1}_{k} -D_{k}^{n+1/2} }{\Delta t}  \,+\,
		\sqrt{\frac{k}{\rho_{\infty}}}\,\mathcal{A}_{h}\,\omega^{n+1}_h
		\left(D^{n+1}_{k-1}\,-\,D_{\infty,k-1}\right)  \,&\ds=\,
	0\,,\quad{\rm if}\,\,k\geq1\,,
	\end{array}\right.
	\ee
for $k\in\{0,\ldots,N_H\}$ and $D_k^{n+1}=0$ when $k>N_H$.
Observe that since $\omega_h^{n+1/2}$ and $D_0^{n+1/2}$ do not change
during this second step, the latter system is trivially invertible and hence does not require
any linear solver. \\
In the next section we analyze the linearized step \eqref{discrete:step1} for which we prove exponential relaxation towards equilibrium in the long time regime. Then, in Section \ref{sec:4}, we will perform numerical simulations on the full nonlinear scheme \eqref{discrete:step1}-\eqref{discrete:step2}.

\section{Trend to equilibrium of the discrete linearized system}
\label{sec:3}
\setcounter{equation}{0}
\setcounter{figure}{0}
\setcounter{table}{0}
In this section, we only consider the numerical scheme applied to the linearized
system \eqref{Hermite:lin} corresponding to the first step in the splitting method \eqref{discrete:step1}-\eqref{discrete:step2}, that is,
\be
	\label{discrete:lin}
	\left\{
	\begin{array}{l}
		\ds
		\eps\,\frac{D^{n+1}_{1} -D_{1}^{n} }{\Delta t}  \,+\,\mathcal{A}_{h}\,D^{n+1}_{0} \,-\,\sqrt{2}\,\mathcal{A}_{h}^\star\,D^{n+1}_{2} \,+\,\cA_h\,\omega_h^{n+1} \,=\,
		-\,  \frac{1}{\tau_0}D^{n+1}_{1}\,,
		\\[1.em]
		\ds
		\eps\,\frac{D^{n+1}_{k} -D_{k}^{n} }{\Delta t}  \,+\,
		\sqrt{k}\,\mathcal{A}_{h}\,D^{n+1}_{k-1}
          \,-\,\sqrt{k+1}\,\mathcal{A}_{h}^\star\,D^{n+1}_{k+1}
          \,=\,-\, \frac{k}{\tau_0}\, D^{n+1}_{k}\,, \,{\rm for }\, k\neq 1\,,\\[1.2em]
		\ds (\cA_h^\star\,\rho_{\infty}^{-1}\,\cA_h) \,\omega_h^{n+1}\, =\, D^{n+1}_0-\sqrt{\rho}_{\infty}\,,
		\\[1.1em]
		\ds\sum_{j\in\J}\Delta x_j \, \omega_j^{n+1} \,\sqrt\rho^{\,-1}_{\infty,j}\,=\,0\,,
	\end{array}\right.
      \ee
for $k\in\{0,\ldots,N_H\}$ and $D_k^{n+1}=0$ when $k>N_H$ and $k=-1$.
Moreover, we define the discrete free energy of the solution $D^n=(D_{k}^n)_{k\in\N}$ to
\eqref{discrete:lin} as follows
\begin{equation}\label{Energy}
	\cE^{n}
	\,=\,
	\frac{1}{2}
	\left(
	\left\|
	D^{n}-D_{\infty}
	\right\|^2_{l^2}
	\,+\,
	\left\|
	\frac{\cA_h\,\omega^n_h}{\sqrt{\rho}_{\infty}}
	\right\|^2_{l^2\left(\T\right)}
	\right)\,,
\end{equation}
where 
\begin{equation*}
	\left\|
	D
	\right\|^2_{l^2}
	\,=\,
	\sum_{k=0}^{N_H} 
	\left\|
	D_k
	\right\|^2_{l^2\left(\T\right)}
	\,,\quad\text{and}\quad
	\left\|
	D_k
	\right\|^2_{l^2\left(\T\right)}
	\,=\,
	\sum_{j\in\cJ} 
	|\cD_{k,j}|^2 \Delta x_j\,,
\end{equation*}
and where $D_{\infty}$ is defined as 
\begin{equation*}
	D_{\infty,k} \,=\,
	\left\{
	\begin{array}{l}
		\sqrt\rho_\infty, \,\, {\rm if } \,\, k=0\,,
		\\[0.9em]
		0, \, \, {\rm else\,,}
	\end{array}
	\right.
\end{equation*}
and recall \cite[Lemma $3.3$]{BF_09_22} that $\cA_h$ and $\cA^\star_h$ are adjoint operators in $l^2(\T)$.\\

Then, we prove that the solution to the fully discrete system
\eqref{discrete:lin} converges exponentially fast to its discrete equilibrium, which is consistent with the continuous system.
\begin{theorem}
        \label{th:3.1}
      Consider the solution $(D^n)_{n\in\N}$ to \eqref{discrete:lin}
      with $N_H\in\N^*$ and $N_x$ an odd number. Then the following discrete energy estimate holds for all $n\geq0$
$$
  \cE^{n} \leq 3\,
  \left(1+\frac{\kappa}{\eps}\,\min\left(\tau_0,\tau_0^{-1}\right)\,\Delta
    t\right)^{-n} \,
    \cE^{0}\,,
$$
where $\kappa>0$ depends only on $\rho_i$, $T_0$ and $|b-a|$.
\end{theorem}
To show this result, we couple a discrete version of the free energy
estimate in Proposition \ref{prop:2.1} with hypocoercive
estimates for the discrete version of the modified relative entropy functional defined in \eqref{eq:H0}.\\

Before to give the proof, let us comment  this  result. On the one hand, we emphasize
that the rate of convergence is uniform with respect to discretization
parameters.
% In particular, this result ensures that the scheme is unconditionally stable, regardless the value of $\eps$.
On the other hand, the convergence rate is proportional to $1/\eps$,
regardless of discretization parameters. This last property ensures that the scheme for the linearized model is asymptotic preserving in the long time regime. It also ensures that our method applied to the linearized model is unconditionally stable, regardless of both scaling and discretization parameters. To conclude, we emphasize that the explicit dependence with respect to $\tau_0$ is coherent with the results for the continuous model \cite{Herda_Rodrigues}.

\subsection{{\it A priori} estimates}
In this section, we prove a discrete free energy
estimate on $(\cE^n)_{n\in\N}$ analogous to the one in Proposition \ref{prop:2.1}. 
\begin{proposition}
  \label{energy:estimate}
Consider the solution $(D^n)_{n\in\N}$ to \eqref{discrete:lin}. The following discrete energy estimate holds for all $n\geq0$
\[
\frac{\cE^{n+1}
	-\cE^{n}}{\Delta t}
\,+\,
\Delta t\,\cR_h^n\,=
\,-\,\frac{1}{\eps\, \tau_0}\,\sum_{k=1}^{N_H} k\,\left\|D^{n+1}_k\right\|^2_{l^2\left(\T\right)}\,,
\]
where $\cR_h^n$ is the following positive remainder due to numeric dissipation
\[
\cR_h^n\,=\,
\frac{1}{2}
\left(
\left\|
\frac{D^{n+1}-D^{n}}{\Delta t}
\right\|^2_{l^2}
+
\left\|
\frac{\cA_h\left(\omega_h^{n+1}-\omega_h^{n}\right)}{\Delta t\,\sqrt{\rho}_{\infty}}
\right\|^2_{l^2\left(\T\right)}
\right)\,.
\]
\end{proposition}
\begin{proof}
	To compute the variations of $\left\|
	D^{n}-D_{\infty}
	\right\|^2_{l^2}$ between time step $n$ and $n+1$, we take the $l^2(\T)$ scalar product between the first line in \eqref{discrete:lin} and $
	D^{n+1}_1-D_{\infty,1}$, the second and $
	D^{n+1}_k-D_{\infty,k}$ and sum over all $k \in \{0,\ldots,N_H\}$, this yields
\begin{equation*}
	\frac{1}{2\,\Delta t}
	\left(
	\left\|
	D^{n+1}-D_{\infty}
	\right\|^2_{l^2}
	-\left\|
	D^{n}-D_{\infty}
	\right\|^2_{l^2}
	+
	\left\|
	D^{n+1}-D^{n}
	\right\|^2_{l^2}\right)
	\,=\,\cI_1\,+\,\cI_2\,,
\end{equation*}
where $\cI_1$ and $\cI_2$ are given by
\be
\left\{
\begin{array}{l}
	\ds \cI_1 :=-\,\frac{1}{\eps\, \tau_0}\sum_{k=0}^{N_H} k\left\|D^{n+1}_k\right\|^2_{l^2\left(\T\right)}
	-\frac{1}{\eps}\sum_{k=0}^{N_H}
	\left\langle\sqrt{k}\,
	\mathcal{A}_{h}\,D^{n+1}_{k-1}
	-\sqrt{k+1} \,\mathcal{A}_{h}^\star D^{n+1}_{k+1},D^{n+1}_{k}-D_{\infty,k}\right\rangle_{l^2\left(\T\right)}
	\,,
	\\[1.3em]
	\ds\cI_2 :=- \frac{1}{\eps}\left\langle
	\cA_h\omega_h^{n+1},D^{n+1}_{1} \right\rangle_{l^2\left(\T\right)}\,,
\end{array}\right.
\ee
where $\cI_2$ stands for the contribution of the electric field and $\cI_1$ gathers all the other terms.\\

First, we rewrite $\cI_1$ using that $\cA_h$ and $\cA_h^\star$ are adjoint in $l^2(\T)$ and that $\cA_h\,D_{\infty,0}=0$ according to \eqref{preserves:equi:h}
\[
 \cI_1 \,=\,-\,\frac{1}{\eps\, \tau_0}\,\sum_{k=0}^{N_H} k\,\left\|D^{n+1}_k\right\|^2_{l^2\left(\T\right)}
\,-\,\frac{1}{\eps}\,\sum_{k=0}^{N_H}
\sqrt{k}\left\langle
\mathcal{A}_{h}\,D^{n+1}_{k-1}
\,,\,D^{n+1}_{k}\right\rangle_{l^2\left(\T\right)}
-\sqrt{k+1}
\left\langle D^{n+1}_{k+1}\,,\,\mathcal{A}_{h} D^{n+1}_{k}\right\rangle_{l^2\left(\T\right)}
\,.
\]
Hence, splitting and re-indexing the second sum in the latter relation, we see that it is in fact zero. Therefore, $\cI_1$ rewrites as follows
\[
\cI_1 \,=\,-\,\frac{1}{\eps\, \tau_0}\,\sum_{k=0}^{N_H} k\,\left\|D^{n+1}_k\right\|^2_{l^2\left(\T\right)}
\,.
\]

Furthermore, considering the case $k=0$ in the second equation of system \eqref{discrete:lin}, we deduce that $\cI_2$ rewrites as follows
\[
\cI_2 \,=\,-\, \left\langle
\omega_h^{n+1}  \,,\,\frac{D^{n+1}_{0} -D_{0}^{n} }{\Delta t} \right\rangle_{l^2\left(\T\right)}\,.
\]
Then, we replace $D^{n+1}_{0} -D_{0}^{n}$ in the latter relation using the third line in \eqref{discrete:lin}, it yields
\[
\cI_2 \,=\,-\, \left\langle
\omega_h^{n+1}  \,,\,(\cA_h^\star\,\rho_{\infty}^{-1}\,\cA_h) \,\frac{\omega_h^{n+1} -\omega_h^{n} }{\Delta t} \right\rangle_{l^2\left(\T\right)}\,.
\]
Using that $\cA_h$ and $\cA^\star_h$ are adjoint in $l^2(\T)$, we deduce the following relation
\[
\cI_2 \,=\,-
\frac{1}{2\,\Delta t}
\left(
\left\|
\frac{\cA_h \,\omega_h^{n+1}}{\sqrt{\rho}_{\infty}}
\right\|^2_{l^2\left(\T\right)}
-\left\|
\frac{\cA_h \,\omega_h^{n}}{\sqrt{\rho}_{\infty}}
\right\|^2_{l^2\left(\T\right)}
+
\left\|
\frac{\cA_h\left(\omega_h^{n+1}-\omega_h^{n}\right)}{\sqrt{\rho}_{\infty}}
\right\|^2_{l^2\left(\T\right)}\right) \,,
\]
which concludes the proof.
\end{proof}

\subsection{Proof of Theorem \ref{th:3.1}}
We are now ready to proceed to the proof of Theorem \ref{th:3.1}. To do so, we develop a discrete hypocoercive technique which consists in introducing the analog $\cH^n$ of the continuous modified entropy functional defined in \eqref{eq:H0}. It reads as follows
\begin{equation}
	\cH^n \,=\,
	\cE^n\,+\, \beta_0\,\left\langle
	\cA_h^\star\, D_1^n, u_h^n\right\rangle_{l^2\left(\T\right)}\,,
	\label{eq:H0:discrete}
\end{equation}
where $\beta_0$ is a positive constant and where $u_h^n$ is solution the solution to
\be\label{eq:elliptic:h}
\left\{
\begin{array}{l}
	\ds \cA_h^\star\,\cA_h \,u_h^{n}\, =\, D_0^n\,-\,\sqrt{\rho}_{\infty}\,,
	\\[1.1em]
	\ds\sum_{j\in\J}\Delta x_j \, u_j^n \,\sqrt\rho_{\infty,j}\,=\,0\,.
\end{array}\right.
\ee
Let us first recall useful estimates on $u^n_h$ \cite[Lemma $3.5$]{BF_09_22}
\begin{lemma}
	\label{lem:3.3}
	For each $N_H\in\N^*$,  $N_x$ an odd number and $n\in\N$, the solution $u^n_h$ to \eqref{eq:elliptic:h} satisfies
	\begin{equation}\label{useful:estimates:h}
		\left\{
		\begin{array}{l}
			\ds \|\cA_h^2\, u^n_h\|_{l^2(\T)}\,+\,\|\cA_h\, u^n_h\|_{l^2(\T)} \,\leq\, C\,\|D_0^n\,-\,\sqrt{\rho}_{\infty}\|_{l^2(\T)}\,,
			\\[1.1em]
			\ds\left\|\cA_h \left(\frac{u_h^{n+1}-u_h^n}{\Delta t}\right)\right\|_{l^2(\T)} \,\leq\, \min{\left(\frac{1}{\eps}\,\left\|D_1^{n+1}\right\|_{l^2(\T)}\,,\,
				C\left\|\frac{D_0^{n+1}-D_0^{n}}{\Delta t}\right\|_{l^2(\T)}
				\right)}\,,
		\end{array}\right.
	\end{equation}
	for some constant $C>0$ only depending on $\rho_i$ and $T_0$.
\end{lemma}
Among other, the latter Lemma ensures that the modified relative entropy functional $\cH^n$ is in fact equivalent to the energy $\cE^n$ for $\beta_0$ small enough.
\bl
For all $\rho_i$ and $T_0>0$, there exists a positive constant $\overline{\beta}_0$ such that for all
$\beta_0\in(0,\overline{\beta}_0)$, one has
\begin{equation}
	\label{disc:equiv}
	\frac{1}{4}\,\cE^n \leq\,\cH^n\,\leq
	\,\frac{3}{4}\,\cE^n\,,\quad \forall\, n\in\N,
\end{equation}
where $\cE^n$ and $\cH^n$ are given by \eqref{Energy} and
\eqref{eq:H0:discrete}  with $N_H\in\N^*$ and $N_x$ is an odd number.
\label{lem:2disc}
\el
\begin{proof}
  Since $N_x$ is an odd number, we know from \cite{BF_09_22} that the
  operator $\cA_h$ satisfies a discrete Poincar\'e inequality. Then, since $\cA_h$ and $\cA^{\star}_h$ are adjoint in $l^2(\T)$, applying Cauchy-Schwarz inequality and the first line in \eqref{useful:estimates:h}, we obtain
\[
\left|
\left\langle
\cA_h^\star\, D_1^n, u_h^n\right\rangle_{l^2\left(\T\right)}\right|
\leq C\,\|D^n\,-\,D_{\infty}\|_{l^2}^2\,,
\]
which allows to bound the additional term in the definition of $\cH^n$ and therefore conclude the proof.
\end{proof}
Building on the latter lemmas, we now prove that $\cH^n$ verifies a dissipation relation
\begin{proposition}
	\label{prop:3.4}
	For  $N_H\in\N^*$ and $N_x$ an odd number, consider the solution $(D^n)_{n\in\N}$ to \eqref{discrete:lin}. The modified relative entropy functional defined by \eqref{eq:H0:discrete} verifies for all $n\geq 0$
	\[
	\frac{\cH^{n+1}
		-\cH^{n}}{\Delta t}\;
	\leq\;
	\,-\,\frac{\kappa}{\eps}\,\min{\left(\tau_0,\tau_0^{-1}\right)}\,\cH^{n+1}
	\,,
	\]
	for some positive constant $\kappa$ depending only on $\rho_i$ and $T_0$ and $|b-a|$.
\end{proposition}
\begin{proof}
We first focus on the additional term $\ds \left\langle
\cA_h^\star D_1^n, u_h^n\right\rangle_{l^2\left(\T\right)}$ in the definition of $\cH^n$.
We write its discrete time derivative as follows
\[
\frac{1}{\Delta t}
\left(
\left\langle
\cA_h^\star D_1^{n+1}, u_h^{n+1}\right\rangle_{l^2\left(\T\right)}
-
\left\langle
\cA_h^\star D_1^n, u_h^n\right\rangle_{l^2\left(\T\right)}
\right)=
\left\langle
\cA_h^\star \frac{D^{n+1}_1- D^{n}_1}{\Delta t}, u_h^{n+1}\right\rangle_{l^2\left(\T\right)}
+
\left\langle
\cA_h^\star D_1^n, \frac{u^{n+1}_h- u^{n}_h}{\Delta t}\right\rangle_{l^2\left(\T\right)}
\]
and replace the discrete time derivative of $D_1^n$ in the latter right-hand side thanks to the first line in \eqref{discrete:lin}
	\begin{equation}\label{time:der:additional:term}
	\frac{1}{\Delta t}
	\left(
	\left\langle
	\cA_h^\star D_1^{n+1}, u_h^{n+1}\right\rangle_{l^2\left(\T\right)}
	\,-\,
	\left\langle
	\cA_h^\star D_1^n, u_h^n\right\rangle_{l^2\left(\T\right)}
	\right)
	\,=\,
	\cJ_1
	\,+\,
	\cJ_2
	\,+\,\cJ_3\,,
	\end{equation}
	where $\cJ_1$, $\cJ_2$ and $\cJ_3$ are given by
	$$
	\left\{\begin{array}{ll}
		\ds\cJ_1\,:=\,
		&\ds-\,\frac{1}{\eps}\,\left\langle\cA_h^\star\cA_h\,\left(D^{n+1}_0-\sqrt\rho_{\infty}\right) 
		,\, u_h^{n+1}\right\rangle_{l^2\left(\T\right)}\,,\\[1.2em]
		\ds \cJ_2\,:=\,&\ds -\,\frac{1}{\eps}\left\langle \cA_h^\star
		\, \cA_h
		\, \omega^{n+1}_h,\,u^{n+1}_h\right\rangle_{l^2\left(\T\right)}
		\,,
		\\[1.3em]
		\ds\cJ_3\,:=\,&\ds
		+\,\frac{1}{\eps}\,\left\langle
		\sqrt{2}\,(\cA_h^\star)^2\, D_2^{n+1}
		\,-\,\frac{1}{\tau_0}\,\cA_h^\star\, D^{n+1}_1
		,\, u_h^{n+1}\right\rangle_{l^2\left(\T\right)}
		\,+\,
		\left\langle \cA_h^\star
		\, D^{n}_1,\,\frac{u^{n+1}_h- u^{n}_h}{\Delta t}\right\rangle_{l^2\left(\T\right)}
		\,.
	\end{array}\right.
	$$
First, $\cJ_1$ is the desired dissipation term: since $\cA_h$ and $\cA^\star_h$ are adjoint in $l^2(\T)$ and according to \eqref{eq:elliptic:h}, it holds
	\[
	\cJ_1\,=\,
	-\,\frac{1}{\eps}\,\left\langle\,D^{n+1}_0-\sqrt\rho_{\infty}
	,\, \cA_h^\star\cA_h\, u_h^{n+1}\right\rangle_{l^2\left(\T\right)}
	\,=\,
	-\,\frac{1}{\eps}\,\left\|D^{n+1}_0-\sqrt\rho_{\infty}\right\|_{l^2\left(\T\right)}^2\,.
	\] 
Then, $\cJ_2$ takes into account the contribution of the electric field. As it turns out, it is also a dissipation term: thanks to \eqref{discrete:lin} (third line) and \eqref{eq:elliptic:h}, we have
	$\ds
	\cA_h^\star\,\cA_h \,u_h^{n+1}\,=\,(\cA_h^\star\,\rho_{\infty}^{-1}\,\cA_h) \,\omega_h^{n+1}
	$, which yields
	\[
	\cJ_2\,=\, -\,\frac{1}{\eps}\left\langle
	 \omega^{n+1}_h,\,(\cA_h^\star\,\rho_{\infty}^{-1}\,\cA_h) \,\omega_h^{n+1}\right\rangle_{l^2\left(\T\right)}\,=\,
	 -\,\frac{1}{\eps}\left\|
	 \frac{\cA_h\,\omega^{n+1}_h}{\sqrt{\rho}_{\infty}}
	 \right\|^2_{l^2\left(\T\right)}\,.
	\]
Finally, $\cJ_3$ gathers all terms without good sign. Since $\cA_h$ and $\cA^\star_h$ are adjoint in $l^2(\T)$, it rewrites 
	\begin{align*}
	\cJ_3\,=\,
	&\frac{\sqrt{2}}{\eps}\,\left\langle
	D_2^{n+1}
	,\, (\cA_h)^2 u_h^{n+1}\right\rangle_{l^2\left(\T\right)}
	-\,\frac{1}{\tau_0\eps}\,\left\langle D^{n+1}_1
	,\, \cA_h\,u_h^{n+1}\right\rangle_{l^2\left(\T\right)}\\[0.8em]
	&+\,
	\left\langle D^{n+1}_1,\,\cA_h\left(\frac{u^{n+1}_h- u^{n}_h}{\Delta t}\right)\right\rangle_{l^2\left(\T\right)}
	+\,
	\left\langle D^{n}_1 - D^{n+1}_1,\,\cA_h\left(\frac{u^{n+1}_h- u^{n}_h}{\Delta t}\right)\right\rangle_{l^2\left(\T\right)}
	\,.
	\end{align*}
	We estimate the first two terms in the latter right-hand side
        using Young inequality and the first line in
        \eqref{useful:estimates:h}. The last two terms are estimated
        applying Cauchy-Schwarz inequality and the second line in
        \eqref{useful:estimates:h}. Hence, we get
	$$
	\cJ_3\,\leq\, \frac{1}{\eps}\,
	\frac{C}{\ols{\tau}_0}\,\eta\, \| D^{n+1}_0-D_{\infty,0}
	\|_{l^2\left(\T\right)}^2 \,+\, \frac{C}{\eta\,\eps}\, \left( \| D^{n+1}_2\|_{l^2\left(\T\right)}^2 \,+\,\frac{1}{\ols{\tau}_0}\, \| D^{n+1}_1\|_{l^2\left(\T\right)}^2\right)
	\,+\,
	\frac{C}{\Delta t}\,
	\| D^{n+1}- D^n\|_{l^2}^2
	\,.
	$$
	for any positive $\eta$, for some positive constant $C$ depending
	only on $\rho_i$ and $T_0$ and with $\ols{\tau}_0=\min(1,\tau_0)$.
Taking the sum between \eqref{time:der:additional:term} multiplied by $\beta_0$ and the estimate in Proposition \ref{energy:estimate} and replacing $\cJ_1$, $\cJ_2$ and $\cJ_3$ with the latter estimates, it yields
	\begin{align*}
		&\frac{\cH^{n+1}
			-\cH^{n}}{\Delta t}
		\,+\,
		\Delta t\,\left(1-C\,\beta_0\right)\cR_h^{n}\,\leq \\
		&-\frac{1}{\eps\,\tau_0}
		\left[
		\left(1-\frac{C}{\eta\,\ols{\tau}_0}\,\beta_0\,\tau_0\right)\sum_{k=1}^{N_H} \left\|D^{n+1}_k\right\|^2_{l^2\left(\T\right)}
		+
		\beta_0\tau_0\left(1-
		\frac{C}{\ols{\tau}_0}\eta\right) \| D^{n+1}_0-D_{\infty,0}
		\|_{l^2\left(\T\right)}^2
		+\beta_0\tau_0\left\|
		\frac{\cA_h\,\omega^{n+1}_h}{\sqrt{\rho}_{\infty}}
		\right\|^2_{l^2\left(\T\right)}
		\right].
	\end{align*}
	Choosing $\eta=\ols{\tau}_0/(2C)$ and
	$\beta_0\tau_0=2\ols{\tau}_0^2/(\ols{\tau}_0^2+4 C^2)$ it holds
	$$
	\left(1-\frac{C}{\eta\,\ols{\tau}_0}\,\beta_0\,\tau_0\right)\sum_{k=1}^{N_H} \left\|D^{n+1}_k\right\|^2_{l^2\left(\T\right)}
	+
	\beta_0\tau_0\left(1-
	\frac{C}{\ols{\tau}_0}\eta\right) \| D^{n+1}_0-D_{\infty,0}
	\|_{l^2\left(\T\right)}^2 \,=\, \frac{\ols{\tau}_0^2}{4\,C^2+\ols{\tau}_0^2}\, \| D^{n+1}-D_{\infty}\|_{l^2}^2, 
	$$
	hence the latter inequality becomes
	\[
	\frac{\cH^{n+1}
		-\cH^{n}}{\Delta t}
	\,+\,
	\Delta t\,\left(1-C\,\frac{2\ols{\tau}_0^2}{\tau_0(\ols{\tau}_0^2+4C^2)}\right)\cR_h^{n}\,\\
	\leq
	\,-\,\frac{1}{\eps\tau_0}\,
	\frac{2\ols{\tau}_0^2}{4\,C^2+\ols{\tau}_0^2}\,\cE^{n+1}
	\,.
	\]
	According to Lemma \ref{lem:2disc} and since $\beta_0\tau_0=2\ols{\tau}_0^2/(\ols{\tau}_0^2+4 C^2)$, we may replace $\cE^n$ with $\cH^n$ for $C\geq 0$ great enough in the latter estimate. Hence, after simple computations, we deduce the result
	\[
	\frac{\cH^{n+1}
		-\cH^{n}}{\Delta t}\;
	\leq\;
	\,-\,\frac{\kappa}{\eps}\,\min{\left(\tau_0,\tau_0^{-1}\right)}\cH^{n+1}
	\,,
	\]
	for some $\kappa>0$ depending only on $\rho_i$ and $T_0$ and $|b-a|$.
\end{proof}
We now conclude the proof of Theorem \ref{th:3.1}. First, from Proposition \ref{prop:3.4}, it is straightforward to obtain
\[
\cH^{n}
\,\leq\,
\left(
1\,+\,\frac{\kappa}{\eps}\,\min{\left(\tau_0,\tau_0^{-1}\right)}\,\Delta t
\right)^{-n}\cH^{0}\,.
\]
Then, we apply Lemma \ref{lem:2disc} on each side of the latter inequality and obtain the result.

\section{Numerical simulations}
\label{sec:4}
\setcounter{equation}{0}
\setcounter{figure}{0}
\setcounter{table}{0}

For numerical experiments,  we apply a slight modification of the
scheme \eqref{discrete:step1}-\eqref{discrete:step2} since  a Strang
splitting scheme with a second order implicit Runge-Kutta scheme is used to get second order accuracy in
time.

In our simulations, we fix the the temperature $T_0$ to $1$ and numerical parameters as follows: $N_x=129$, $\Delta t = 0.1$ and we adapt the number of Hermite modes
depending on the collisional regime. In Section \ref{sec:AP} we highlight the asymptotic preserving properties of our scheme and investigate the behavior of solutions when $\eps\ll 1$ and $\tau_0$ is fixed. Then, from Section \ref{sec:pert:density} to Section \ref{sec:two:stream}, we will fix $\eps=1$ in \eqref{discrete:step1}-\eqref{discrete:step2} and investigate the robustness of the scheme in different collisional regimes ranging from weakly collisional regime when $\tau_0\gg 1$
to strongly collisional plasmas when $\tau_0\simeq 1$. 

%In the numerical simulations presented in this section, we fix $T_0=1$
%and $\eps=1$ and take different values for the collisional frequency
%$\tau_0$ describing either weakly collisional regime when $\tau_0\gg 1$
%or strongly collisional plasmas when $\tau_0\simeq 1$. Then, we choose
%$N_x=129$ and $\Delta t = 0.1$ and adapt the number of Hermite modes
%depending on the collisional regime.

\subsection{Asymptotic-preserving properties}\label{sec:AP}

In this first test, we illustrate the robustness and the asymptotic preserving property of the scheme in the limit $\eps\rightarrow 0$. To do so, we keep $\Delta t = 0.1$ fixed and perform numerical
simulations with $\eps=10^{-k}$, for $k$ ranging from $0$ to
$6$. Since $\eps$ only appears in front of the time derivative, it can
be interpreted as a time scaling parameter. We emphasize that the extreme case $\eps=10^{-6}$ corresponds to taking
time step of the order {$\Delta t /\eps = 10^{5}$} : %we do not expect
                                %our scheme to reach  numerical
                                %convergence for such parameters but
we expect that the solution approaches  $f_\infty$ in one time step!
In this study, the numerical parameters are fixed and we  only consider a few number
of Hermite modes taking $N_H=80$. \\
We  choose the following spatially inhomogeneous equilibrium
$$
\rho_\infty(x) \,=\, c_\infty\, \exp\left({-\phi_\infty(x)}\right),\quad x \in (-L,L)\,, 
$$
where the potential $\phi_\infty$ is given by
$$
\phi_\infty(x) \,=\, 0.2\,\sin(k\,x)\,, \quad x \in (-L,L),
$$
with $k=\pi/L$ and $L=6$  whereas the constant $c_\infty$  ensures that
$$
\frac{1}{2L} \int_{-L}^L \rho_\infty(x) \,\dD x \,=\, 1.
$$
Thus, we take the initial distribution function as a perturbation of
this steady state, that is, 
$$
f(t=0,x,v) \,=\, \frac{1}{\sqrt{2\pi}}\,\left(\rho_\infty(x)\,+\,\delta\, \cos(k\,x)\right)\,\, \exp\left(-\frac{v^2}{2}\right)\,, \quad (x,v) \in (-L,L)\times \R,
$$
where $\delta=0.01$ and consider the case $\tau_0=10^5$, which corresponds to a
weakly collisional regime. 

\begin{figure}[htbp]
	\begin{tabular}{cc}
		\includegraphics[width=8.cm,angle=0]{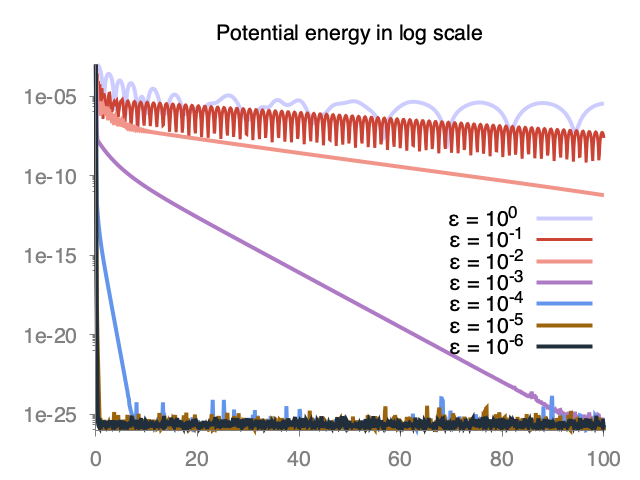} &
		\includegraphics[width=8.cm,angle=0]{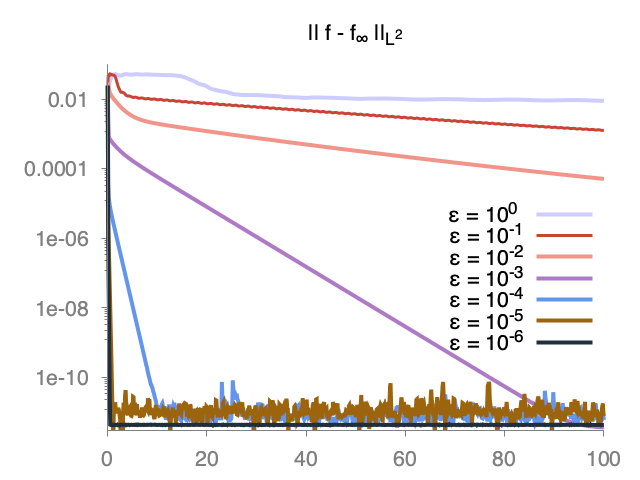} \\
		\\
		(a) & (b)
	\end{tabular}
	
	\caption{ {\it Asymptotic-preserving properties: time development of
			(a) the potential energy (b) $\| f -f_\infty
			\|_{L^2\left(f_\infty^{-1}\right)}$ (in log scale).}}
		\label{fig:-1}
\end{figure}

On Figure \ref{fig:-1}, we represent the time evolution of
the potential energy and  the quantity
\begin{equation}
	\label{def:L2}
	\cL_2(t)\, :=\,
	\|f(t)-f_\infty\|_{L^2\left(f_\infty^{-1}\right)}.
\end{equation}

As predicted by our analysis of the linearized model, we observe
exponential relaxation towards equilibrium at a rate which is
proportional to $1/\eps$. The scheme is uniformly stable 
with respect to $\eps$ and the solution converges to the discrete
equilibrium $f_\infty$ when $\eps\rightarrow 0$.

\begin{figure}[htbp]
\begin{tabular}{cc}
\includegraphics[width=8.cm,angle=0]{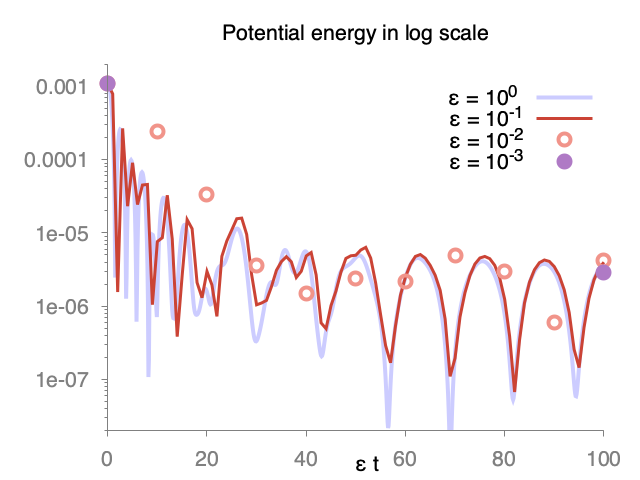}&
\includegraphics[width=8.cm,angle=0]{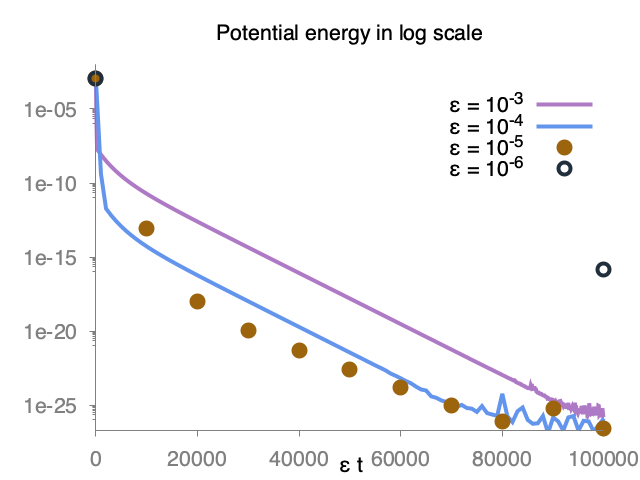}
\end{tabular}
\caption{ {\it Asymptotic-preserving properties: re-scaled time development $(s\leftarrow \eps\,t)$ of the potential energy (in log scale).}}
\label{fig:0}
\end{figure}

The left chart in Figure \ref{fig:0} represents the time evolution of the potential energy in log-scale for $t/\eps\in[0,100]$. We plot the approximations corresponding to $\eps=10^{-k}$, for $k$ ranging from $0$ to $3$ and, in each case, re-scale the time variable as $(s\leftarrow t/\eps)$ to compare solutions. We can see that the case $\eps=10^{-1}$ (which corresponds to taking a time step $\Delta t /\eps = 1$) fits very well with the approximation obtained when $\eps =1$. Both the first phase for $t/\eps\in[0,20]$ corresponding to fast oscillations and steep descent and the second phase for $t/\eps\in[20,100]$ corresponding to slower oscillations without damping are transcribed correctly in this case. When $\eps=10^{-2}$, we do not expect the solution to be precise during the first phase $t/\eps\in[0,20]$ since the corresponding time step $\Delta t /\eps = 10$, is greater than the time period of the oscillations. However, we see that in the second phase $t/\eps\in[20,100]$, the approximation catches up and even seems to capture the oscillatory behavior of the solution, despite the fact that time step and oscillation period have the same order of magnitude ($\simeq 10$). To conclude, the case $\eps=10^{-3}$ corresponds to taking a time step $\Delta t /\eps = 100$. Therefore, the approximation at time $t/\eps=100$ is surprisingly accurate considering that it was calculated in only a single iteration !

The right chart in Figure \ref{fig:0} represents the time evolution of
the potential energy in log-scale for $t/\eps\in[0,10^{5}]$. We plot
the approximations corresponding to $\eps=10^{-k}$, for $k$ ranging
from $3$ to $6$ and once again, in each case, re-scale the time
variable as $(s\leftarrow t/\eps)$ to compare solutions. We do not
observe oscillations anymore since collisions effects take over
transport phenomena at this time scale. Therefore, we compare the
exponential decay rate of our approximations in order to validate at
the nonlinear level our theoretical result, which holds for the
linearized scheme. We observe that in all cases $k=3,\dots,6$,
approximations present similar decay rate, even when $k=6$, which is
surprising considering the fact that in this case the approximation at
time $t/\eps=10^{5}$ was obtained in one iteration only. All these
results are  very satisfying! Indeed,  for the
differents regimes, corresponding to the values of $\eps$, our
numerical schemes is able to describe correctly the different phases:
an oscillatory behavior when $\eps\geq 10^{-3}$ and an exponential decay to
equilibrium when $\eps\leq 10^{-3}$.

To conclude this section, it is worth to mention that in this test only, we tuned the expert options of the Super
LU library \cite{superlu} used as a direct solver in our code. More
precisely, for small values of $\eps$, the system associated to
\eqref{discrete:step1} may be ill conditioned, hence we disabled the equilibration option and tuned the threshold used for a diagonal entry to be an acceptable pivot in the factorization.

\subsection{Perturbation of non uniform density}\label{sec:pert:density}
For this second numerical test, we  consider the same initial
condition as in the preceding Section \ref{sec:AP}. However, we now fix $\eps$ to $1$ and perform simulations with variable $\tau_0$. To enforce numerical convergence, we have chosen a large number of Hermite modes $N_H =400$ when the plasma is weakly
collisional, that is when $\tau_0\gg 10^{2}$, since filamentation may occur in phase
space whereas $N_H=50$ is enough when collisions dominate.

On the one hand, we take $\tau_0=10^4$ corresponding to the weakly
collisional regime and compare two solutions, one is obtained using
\eqref{discrete:lin} corresponding to the linearized
Vlasov-Poisson-Fokker-Planck system \eqref{vpfp:lin0} and the second one
is given by \eqref{discrete:step1}-\eqref{discrete:step2}
corresponding to the nonlinear Vlasov-Poisson-Fokker-Planck system
\eqref{vpfp:0}. Our results show that  both solutions have the same
behavior, which means that, for such a small perturbation, the linear
regime governs the dynamics. To illustrate this observation, we report in
Figure \ref{fig:2.1} the time evolution of the potential energy
$$
\cE_p(t) := \int_{\T} \left|\partial_x \psi(t,x)\right|^2\,\dD x
$$
obtained using \eqref{discrete:lin} in $(a)$
and \eqref{discrete:step1}-\eqref{discrete:step2} in $(b)$. Both
solutions first produce fast damped oscillations up to time $t\leq
20$ and then oscillate with a lower frequency while converging
exponentially fast to zero with the same convergence rate $\gamma_L\simeq
0.004$.

\begin{figure}[htbp]
\begin{tabular}{cc}
\includegraphics[width=8.cm,angle=0]{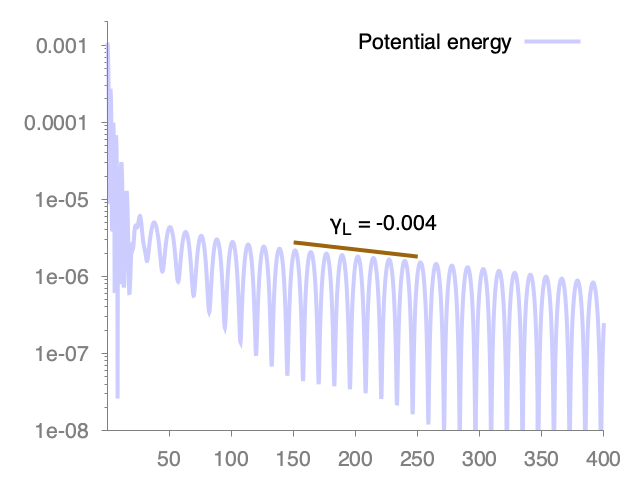} &
\includegraphics[width=8.cm,angle=0]{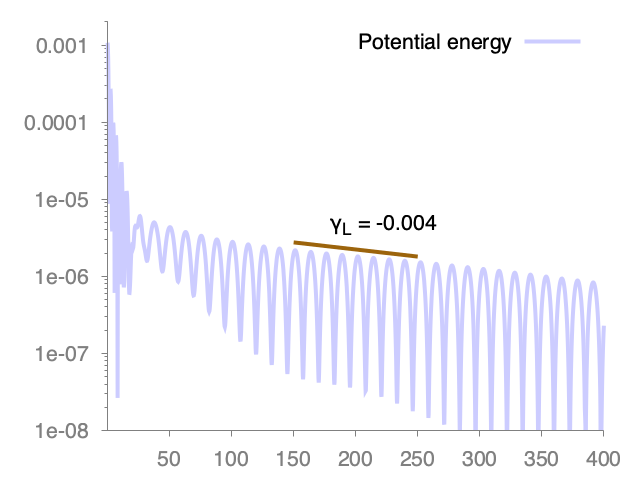} \\
\\
(a) & (b)
\end{tabular}
\caption{ {\it Perturbation of non uniform density for $\tau_0=10^{4}$ (weakly collisional regime): time development of the potential energy in log scale (for (a) the linearized Vlasov-Poisson-Fokker-Planck system and (b) the nonlinear Vlasov-Poisson-Fokker-Planck system.}}
\label{fig:2.1}
\end{figure}

 In Figure \ref{fig:2.2}, we show several snapshots of the
difference between the distribution function $f$ and its equilibrium
$f_\infty$ for $t\in[4,70]$. As expected, thin filaments propagate in
phase space for large velocities but surprisingly we also observe that a vortex is generated in the region where
$|v|\leq 1$. For large time, this vortex remains and continues to
rotate around the point $(x_C,v_C)=(-3,0)$. For such a regime, where collisions are almost negligible, the
amplitude of the perturbation does not vanish even when $t\simeq 70$ and transport phenomena dominate.

\begin{figure}[htbp]
\begin{tabular}{cc}
\includegraphics[width=8.cm,angle=0]{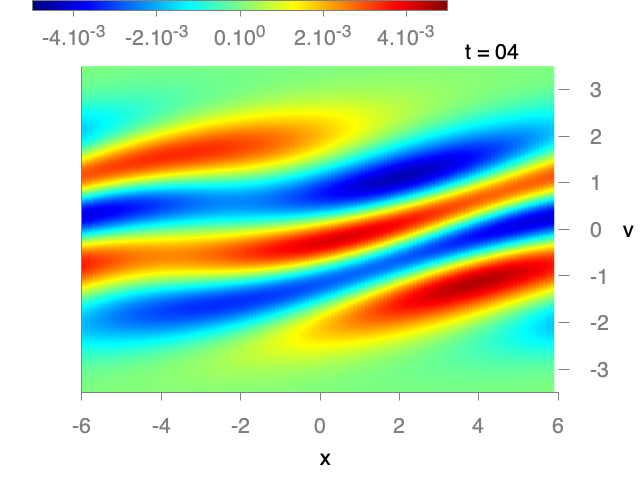}&
\includegraphics[width=8.cm,angle=0]{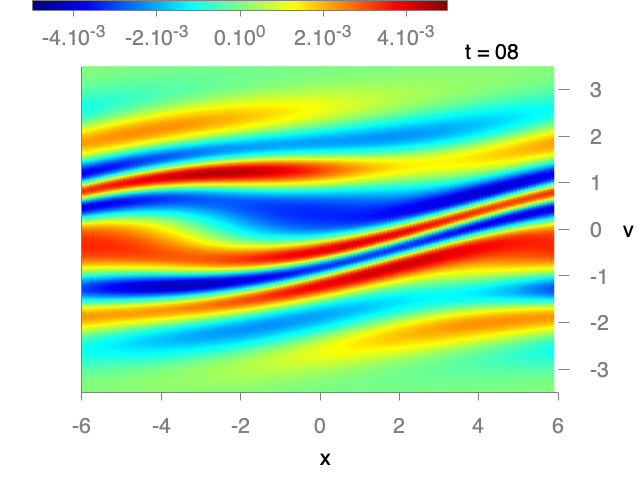}
 \\
 \includegraphics[width=8.cm,angle=0]{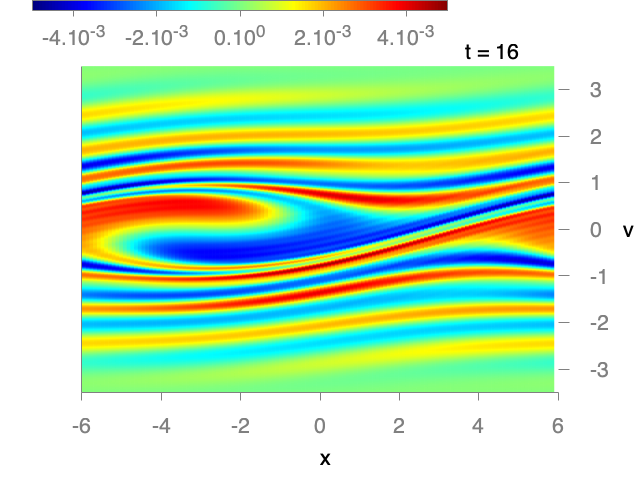}&
\includegraphics[width=8.cm,angle=0]{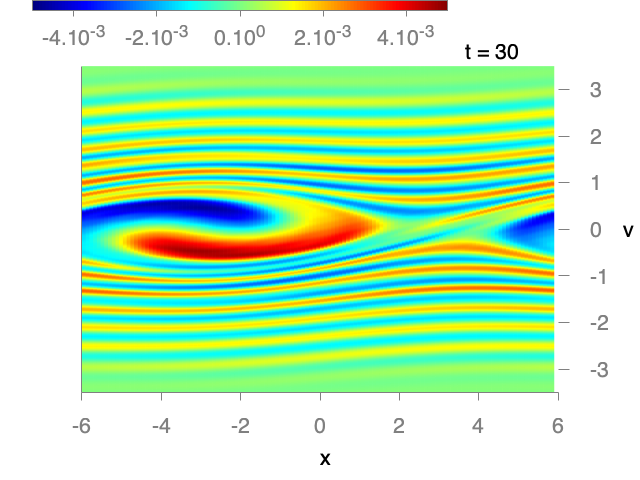}
 \\
 \includegraphics[width=8.cm,angle=0]{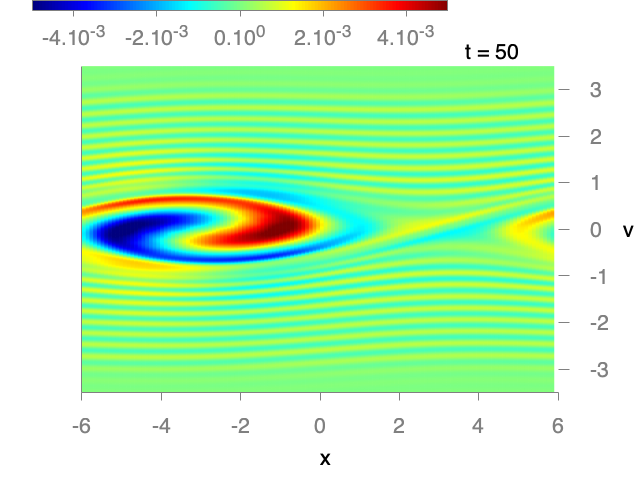}&
\includegraphics[width=8.cm,angle=0]{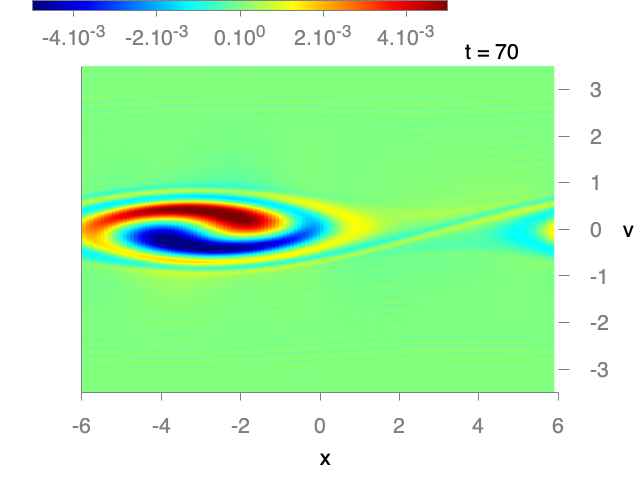}
\end{tabular}
\caption{ {\it Perturbation of non uniform density (weakly collisional regime, $\tau_0=10^{4}$): snapshots of the difference between the solution $f$ and the equilibrium $f_\infty$ at time $t=4$, $8$, $16$, $30$, $40$ and $70$.}}
\label{fig:2.2}
\end{figure}

On the other hand, we study the influence of the collision frequency $\tau_0$ and
perform several numerical simulations for the nonlinear
Vlasov-Poisson-Fokker-Planck system \eqref{vpfp:0} with the same initial data for $\tau_0=10^k$, with $k=0,\ldots,4$ (see Figure \ref{fig:2.3}). For a weakly collision regime, that is $k \geq 3$, we again observe oscillations of the potential
energy and an exponential decay. However, when collisions dominate, fast
oscillations only occur for short time, then the potential energy
decays rapidly to zero without any oscillations. This trend to
equilibrium can be also viewed on the distribution function as shown on Figure
\ref{fig:2.3}-(b), where we present the time evolution of the quantity
$\cL_2$ defined in \eqref{def:L2}.

\begin{figure}[htbp]
\begin{tabular}{cc}
\includegraphics[width=8.cm,angle=0]{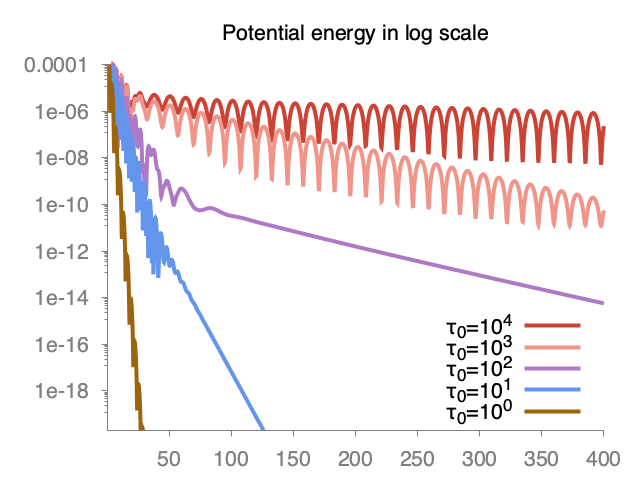}&
\includegraphics[width=8.cm,angle=0]{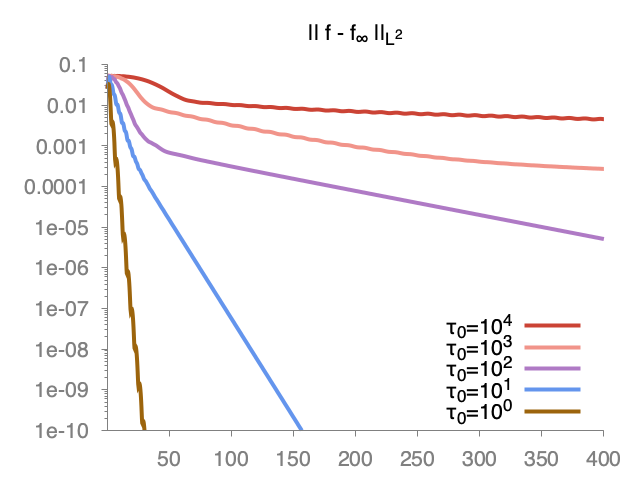}
\end{tabular}
\caption{ {\it Perturbation of non uniform density: time development of (a) the potential energy (b) $\| f -f_\infty \|_{L^2\left(f_\infty^{-1}\right)}$ for various $\tau_0=1,\ldots, 10^4$ (in log scale).}}
\label{fig:2.3}
\end{figure}

Finally in Figure \ref{fig:2.4}, we again present snapshots of
$f-f_\infty$ at different time when $\tau_0=10^2$. In this situation,
collisions dissipate thin filaments generated by the transport term
and the amplitude of the perturbation vanishes. As a consequence, we
do not detect anymore the vortex structure on the difference
$f-f_\infty$ and the solution converges exponentially fast to the
equilibrium  as it has been shown in Theorem \ref{th:3.1} for the
linearized system.

\begin{figure}[htbp]
\begin{tabular}{cc}
\includegraphics[width=8.cm,angle=0]{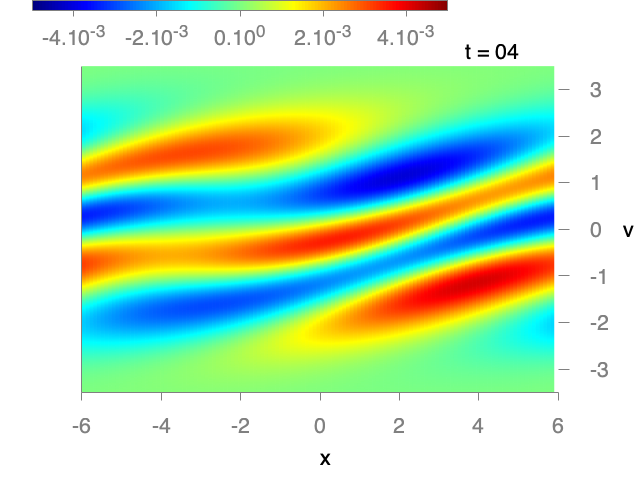}&
\includegraphics[width=8.cm,angle=0]{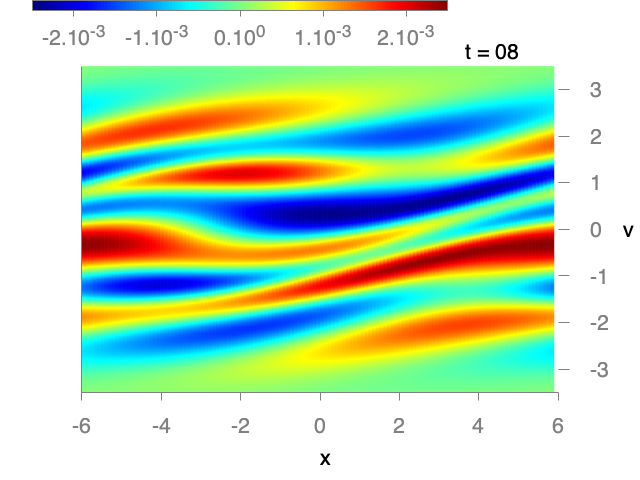}
  \\
\includegraphics[width=8.cm,angle=0]{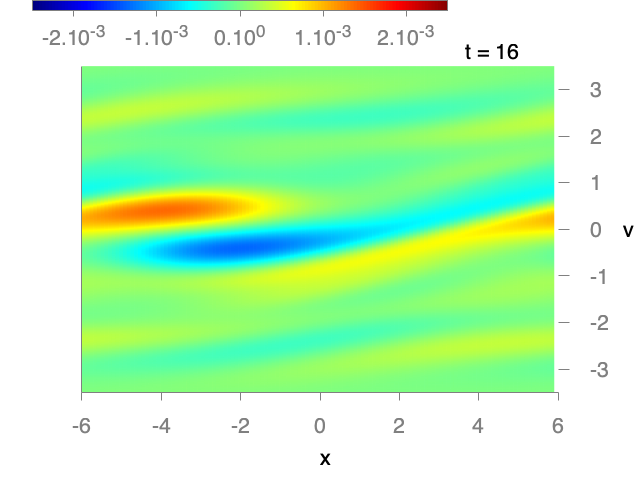}&
\includegraphics[width=8.cm,angle=0]{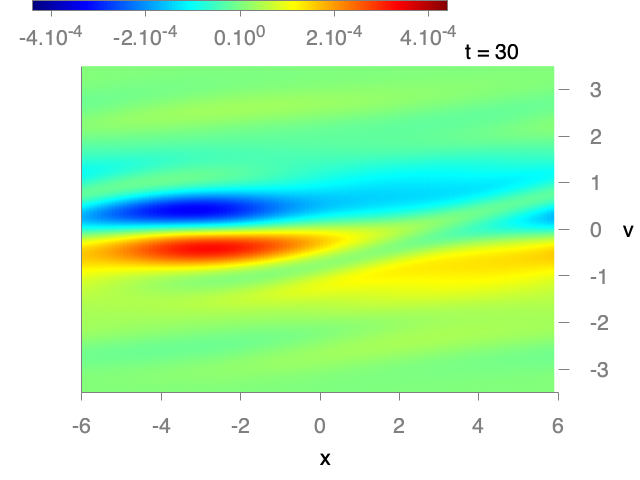}
  \\
 \includegraphics[width=8.cm,angle=0]{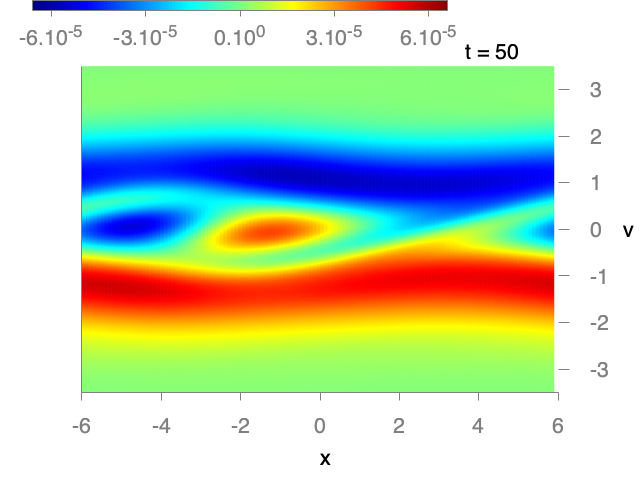} &
\includegraphics[width=8.cm,angle=0]{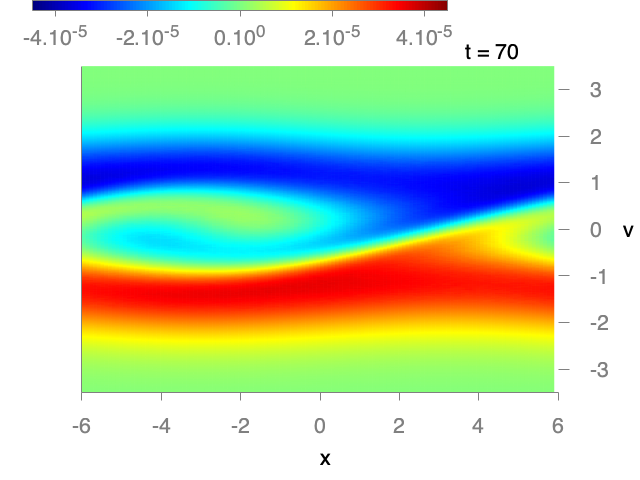}
\end{tabular}
\caption{ {\it Perturbation of non uniform density (moderate collisional regime, $\tau_0=10^{2}$): snapshots of the difference between the solution $f$ and the equilibrium $f_\infty$ at time $t=4$, $8$, $16$, $30$, $40$ and $70$.}}
\label{fig:2.4}
\end{figure}

\subsection{Plasma echo}

We now investigate a much more intricate problem where the non-linearity plays
the main role. Following the work \cite{echo1,echo2}  or more
recently \cite{Bedrossian17,Grenier_Toan_Rodnianski22}, we will consider a
perturbation of an homogeneous Maxwellian distribution $f_{\infty}(x,v):=\cM(v)$ where 
$$
\cM(v) =
\frac{1}{\sqrt{2\pi}}\exp\left(-\frac{v^2}{2}\right), \quad v \in\R\,,
$$
on the space
interval $[-L,\,L]$, with $L=6$.
Of course, this homogeneous Maxwellian is stable: for a small
perturbation in high order Sobolev norms, we expect to observe the well
known Landau damping on the macroscopic quantities. However, our aim
is to investigate a transient regime where a plasma echo occurs for a
well chosen perturbation. This echo appears when two waves with
distinct frequencies interact. For a large time period, the
effect is very small but at certain particular times, the
interaction becomes strong: this is known in plasma physics as the plasma echo,
and can be thought of as a kind of resonance \cite{echo1,echo2}.

In all this section, we fix $\eps$ to $1$. To build our initial
condition, we proceed in two steps \cite{filbetCPC}. We first solve numerically the Vlasov-Poisson system with almost no collisions ($\tau_0=10^6$) on the time interval $[-30, 0]$ with an initial data at time $t_0=-30$ given by
$$
\tilde f_{in}(x,v) \,=\, \left(1+\delta \, \cos(k_1 x)\right)\, \cM(v)\,,
$$ 
where $\delta =0.01$ and $k_1=\pi/L$. This choice induces an
exponential decay of the potential
energy by Landau damping at the rate associated to the perturbed mode $k_1$, hence
it gives a distribution function  which is almost space homogeneous
with small and fast oscillations in  velocity. Then, at $t=0$, we consider
this solution, denoted by $\tilde f(0)$, and choose it as initial data
with  a  perturbation of the mode $k_2:=2k_1$. More precisely, we take
$$
f_0(x,v) \,=\,  \tilde f(0,x,v) \,+\,\delta \, \cos(k_2 x)\, \cM(v) \,.
$$
This initial data is represented in Figure \ref{fig:1.3}. Then, starting from $f_0$, we solve the
Vlasov-Poisson-Fokker-Planck system on the time interval $[0,120]$.

On the one hand, we take $\tau_0=10^{6}$, which corresponds to a
weakly collisional regime generating an oscillatory solution in
velocity. For this reason, we choose a large number of Hermite modes $N_H=8000$ in such a way that our
numerical results are not anymore  sensitive to the numerical parameters. Let
us emphasize that we compare our numerical results with those
obtained using a semi-Lagrangian method
\cite{filbetCPC,sonnen,Filbet2001}, which also requires such a large
number of points in order to reach convergence. The phenomena is so intricate that we want to be sure that numerical parameters do not
produce any artefact...

Now let us comment our numerical results. In this weakly collisional regime, we expect that this initial data
will induce a first Landau damping phenomena due to the perturbation
of the second mode $k_2$. However, after a time much longer than the
inverse Landau damping rate, a new wave, called "echo" will appear as
a modulation of the density at the wave number $k_{echo} = k_2-k_1=
k_1$. This echo is due to the interaction between modes and is
essentially a phenomenon of beating between two waves. First, we will
see that the nonlinearity is crucial here. Indeed, in Figure
\ref{fig:1.1}, we show the time evolution of the potential energy and the first modes of
the electric field  obtained by applying the scheme \eqref{discrete:lin} corresponding to the linearized Vlasov-Poisson-Fokker-Planck
system and the scheme \eqref{discrete:step1}-\eqref{discrete:step2}
corresponding to the nonlinear system.  The numerical solution corresponding to the linearized system exhibits
a simple Landau damping, when $t\geq 5$, with a decay rate corresponding to the 
predicted value $\gamma_L = 0.355$, whereas the numerical solution
corresponding to the nonlinear system differs completely. It exhibits a
first fast decay  as for the linearized system, but when $t\geq 15$,
it produces an echo on the potential energy and the first mode (see
right-hand side of Figure \ref{fig:1.1}). The echo reaches its maximal amplitude at
$t=30$ for which we report the snapshots of the $f-f_\infty$ in Figure
\ref{fig:1.3}.  The first damping of the
perturbed mode $k_2$ for short time $t\leq 5$ and the subsequent echo are
accurately reproduced. From \cite{filbetCPC},  the echo wave number
is indeed expected to be $k_{echo} = k_1$ the first echo time
is predicted at 
$$
t_{echo}  \,=\, t_0+\frac{k_2}{k_2-k_1}(0-t_0)\,=\,30\,,
$$ 
which corresponds very well with the numerical value we obtain here. From time $t=0$
to $t\simeq 20$, the second wave corresponding to the mode $k_2$ has no
effect on the first mode $k_1$ of the electric field, but at time
$t=30$, it is strongly perturbed by the echo effect. Actually, our
numerical results illustrate the complexity of these phenomena.

On the one hand,   we notice that \textbf{echoes are repeated  through time}. On the
potential energy (see the top right chart of Figure \ref{fig:1.1}), a Landau
damping  is observed when $30\leq t\leq 80$, then we again discern a new echo around time
$t=90$ followed by a new damping. We also remark that this second damping ($t\geq 90$) unfolds with a smaller decay rate than the first one ($30\leq t\leq 80$). This repetition
of echoes can be also perceived on the modes of the electric field (see the bottom right chart of Figure \ref{fig:1.1}).

On the
other hand,  we report the time evolution of the first modes of the
electric field (see the bottom right chart of Figure \ref{fig:1.1}) and notice that \textbf{other modes are also subjected to echoes} but at times which differ from the "macroscopic" echo time $t_{echo}=30$. These echoes are not visible on the potential energy since their amplitude is smaller than the one of the potential energy by several order of magnitude. However, we point out that the third mode is subjected to a dramatic echo whose maximal amplitude, reached at time $t=15$, is greater than the initial amplitude by a factor almost $10^5$. A careful inspection allows to distinguish the effect of this echo on the overall amplitude of the potential energy (compare bottom and top right charts of Figure \ref{fig:1.1}). It
is worth to mention that all modes corresponding to the linearized
system are subjected to the classical Landau damping without any
echo. This is a nice example where nonlinear effects, even small, induce intricate oscillatory behaviors.

\begin{figure}[htbp]
   \begin{tabular}{cc}
    \includegraphics[width=8.cm,angle=0]{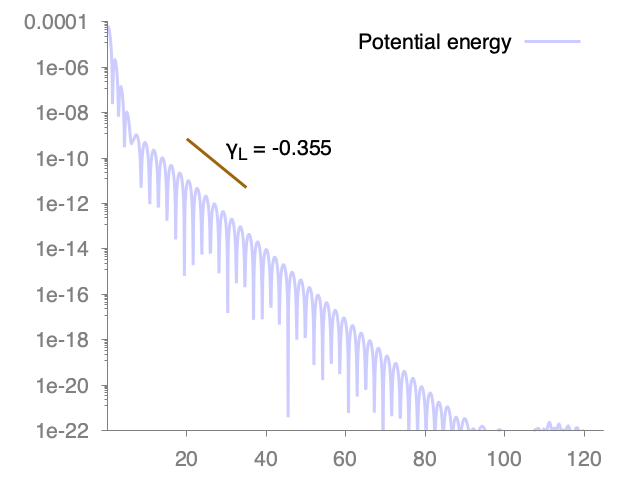} &
    \includegraphics[width=8.cm,angle=0]{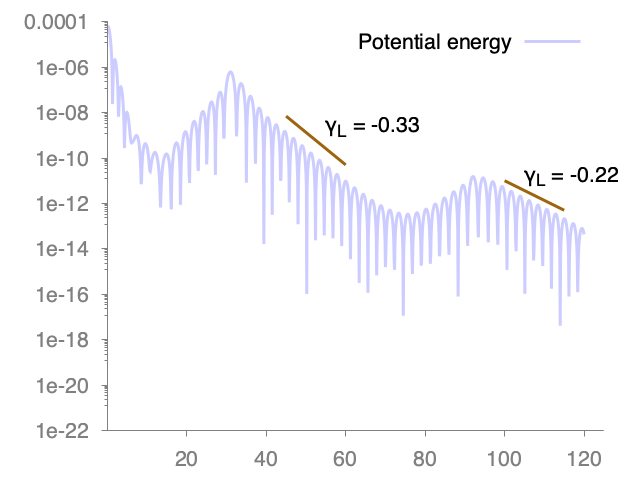} \\
  \includegraphics[width=8.cm,angle=0]{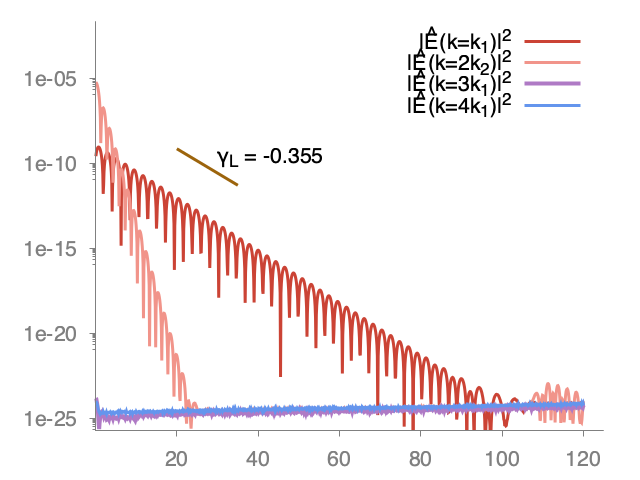} &
    \includegraphics[width=8.cm,angle=0]{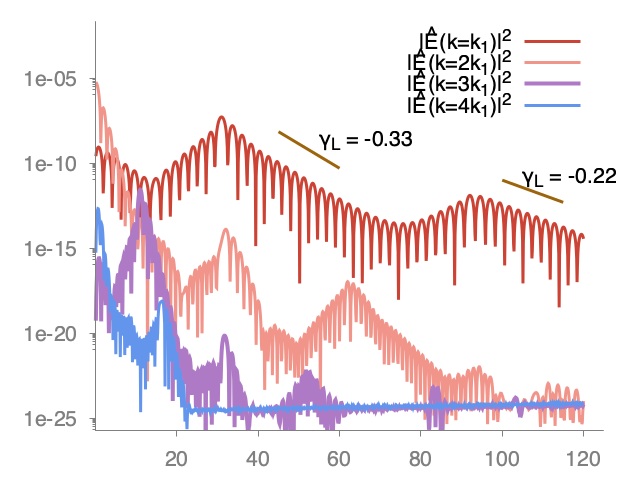}
    \\
    (a) & (b)
\end{tabular}
\caption{ {\it Plasma echo for $\tau_0=10^{6}$ (weakly collisional regime): time development of the potential energy (top) and square of the $k$-th mode of the electric field for $k=k_1,...,4k_1$ in log scale (bottom) for (a) the linearized Vlasov-Poisson-Fokker-Planck system and (b) the nonlinear Vlasov-Poisson-Fokker-Planck system.}}
\label{fig:1.1}
\end{figure}

For this weakly collisional regime ($\tau_0=10^{6}$), we also report
(on Figure \ref{fig:1.2}) the time evolution of the quantity
$$
\cL_2(t)\, :=\, \|f-f_\infty\|_{L^2\left(f_\infty^{-1}\right)}.
$$
First, we point out that unlike for potential energies, we do not observe any difference between the behavior of the linearized \eqref{discrete:lin} and the nonlinear \eqref{discrete:step1}-\eqref{discrete:step2} solutions at the level of distribution functions on these charts. Second, we notice that on this time interval, collisions are negligible and we clearly see that the distribution function $f$ does not yet converge to $f_\infty$. Figure \ref{fig:1.2} also shows that unlike in the case of strong Landau damping, variations of the spatial distribution occurs at an amplitude which is way smaller than the error between kinetic distributions. 
\begin{figure}[htbp]
   \begin{tabular}{cc}
    \includegraphics[width=8.cm,angle=0]{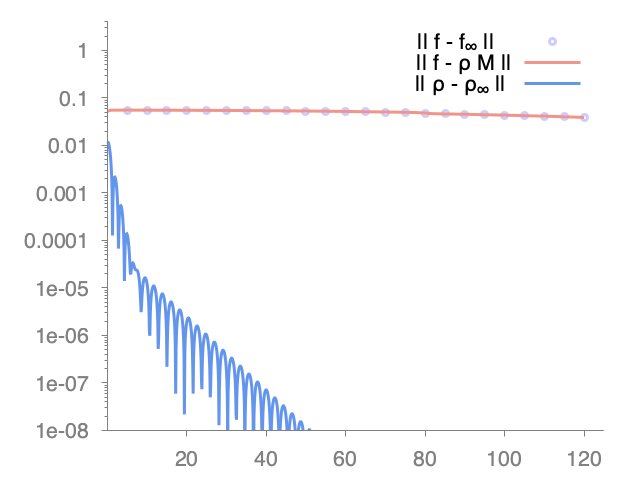} &
    \includegraphics[width=8.cm,angle=0]{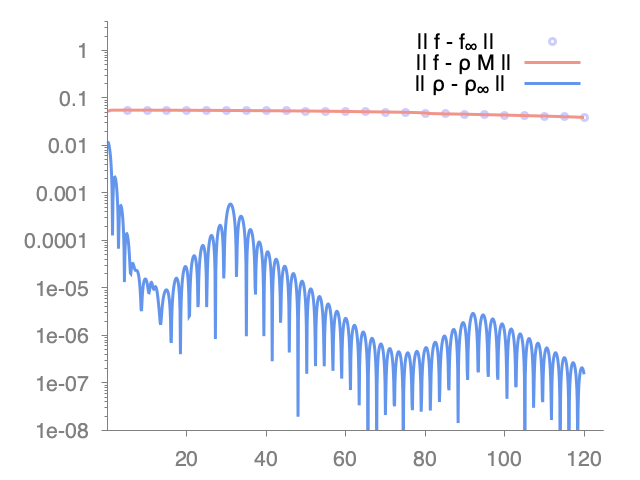} \\
     (a) & (b)
\end{tabular}
\caption{ {\it Plasma echo for $\tau_0=10^{6}$ (weakly collisional regime): time development of $\ds\|f-f_\infty\|_{L^2(f_\infty^{-1})}\,$,\\ $\ds\|f-\rho\,\cM\|_{L^2(f_\infty^{-1})}$ and $\ds\|\rho-\rho_\infty\|_{L^2(\rho_\infty^{-1})}$ in log scale for (a) the linearized Vlasov-Poisson-Fokker-Planck system and (b) the nonlinear Vlasov-Poisson-Fokker-Planck system.}}
\label{fig:1.2}
\end{figure}

This can be also viewed in Figure \ref{fig:1.3}, where we report the snapshots of the difference
between the distribution function $f$ solution to the nonlinear system \eqref{discrete:step1}-\eqref{discrete:step2} and the equilibrium $f_\infty$. We first
observe the projection of the initial data which exhibits
oscillations in velocity and a smooth perturbation in $x$ with a small 
amplitude of order $10^{-3}$. At time $t=30$ when the echo occurs, we
clearly see that the first mode $k_1=1$ dominates, then the solution
continues to oscillate due to the transport operator in a periodic
domain in space.

\begin{figure}[htbp]
\begin{tabular}{cc}
\includegraphics[width=8.cm,angle=0]{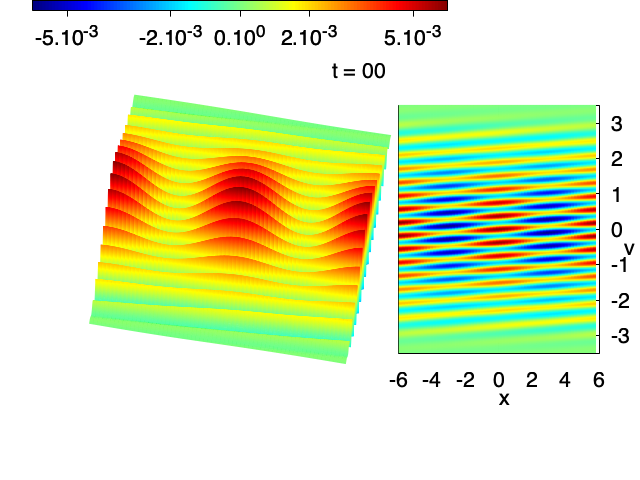}&
\includegraphics[width=8.cm,angle=0]{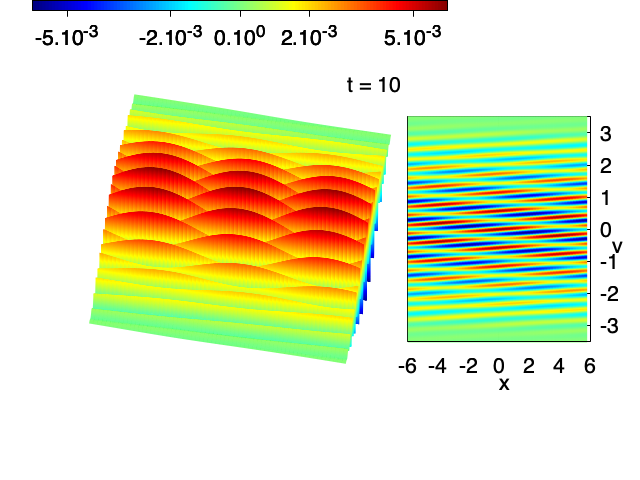}
\\
\includegraphics[width=8.cm,angle=0]{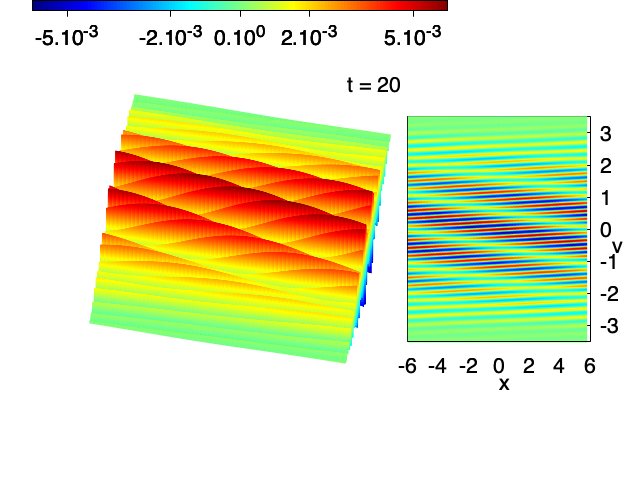}&
\includegraphics[width=8.cm,angle=0]{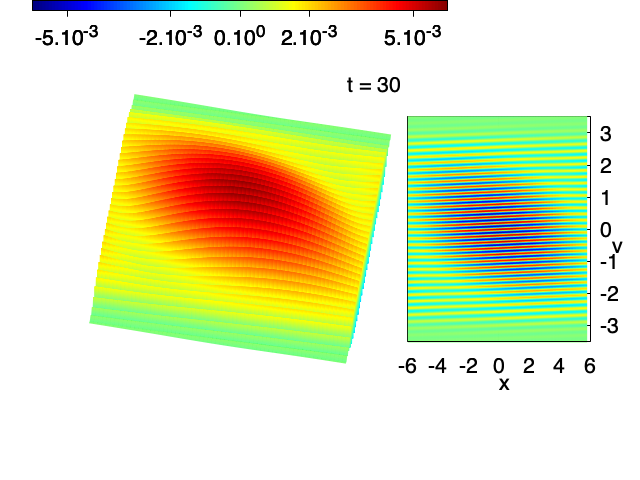}
\\
\includegraphics[width=8.cm,angle=0]{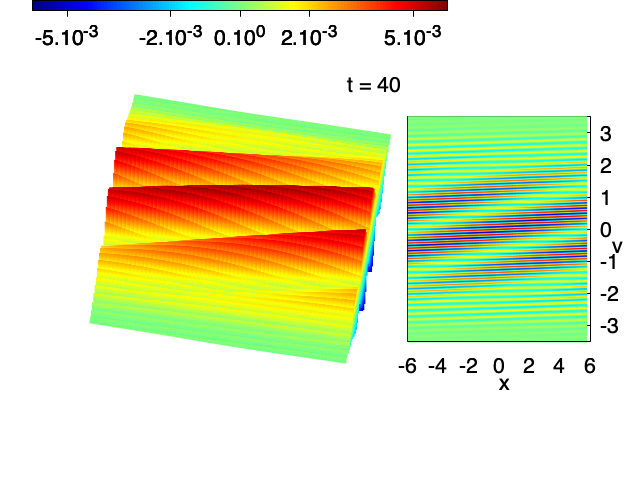}&
\includegraphics[width=8.cm,angle=0]{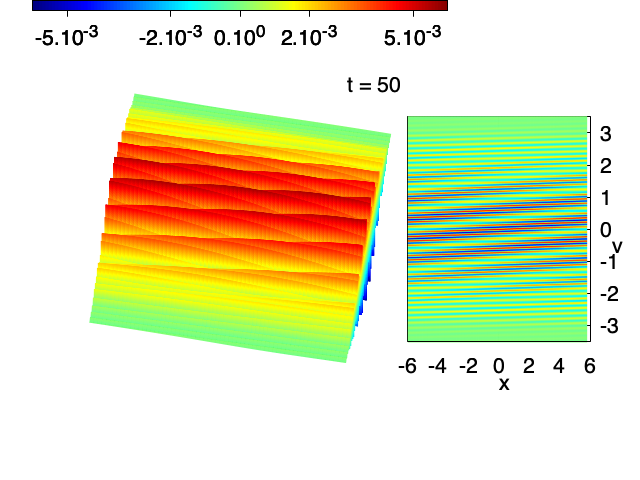}
\end{tabular}
\caption{ {\it Plasma echo for $\tau_0=10^{6}$ (weakly collisional regime): snapshots of the difference between the solution $f$ and the equilibrium $f_\infty$ at time $t=0$, $20$, $30$, $40$ and $50$.}}
\label{fig:1.3}
\end{figure}

A natural question in physics is "how to cancel plasma
echo?" for which a natural answer is that collisions may play a role, as shown in several recent articles \cite{Bedrossian17,Chaturvedi_Luk_Nguyen23}. To illustrate this phenomenon, we perform new numerical simulations passing from weakly to strongly
collisional regime and again compare the two solutions corresponding
to the 
the linearized system \eqref{discrete:lin} and the nonlinear one
\eqref{discrete:step1}-\eqref{discrete:step2}. The results are now reported in Figure \ref{fig:1.4}. Roughly speaking, when $\tau_0> 10^2$, the nonlinear system exhibits a plasma echo whereas when collisions dominate, the electric field is rapidly damped and the
solution converges to its equilibrium $f_\infty$ exponentially fast as predicted by our analysis. It is worth to mention again that at the level of the distribution function, the $\cL_2$ time evolution of linearized \eqref{discrete:lin} and nonlinear \eqref{discrete:step1}-\eqref{discrete:step2} solutions are globally the same for various regimes independently of the plasma echo.
\begin{figure}[htbp]
\begin{tabular}{cc}
 \includegraphics[width=8.cm,angle=0]{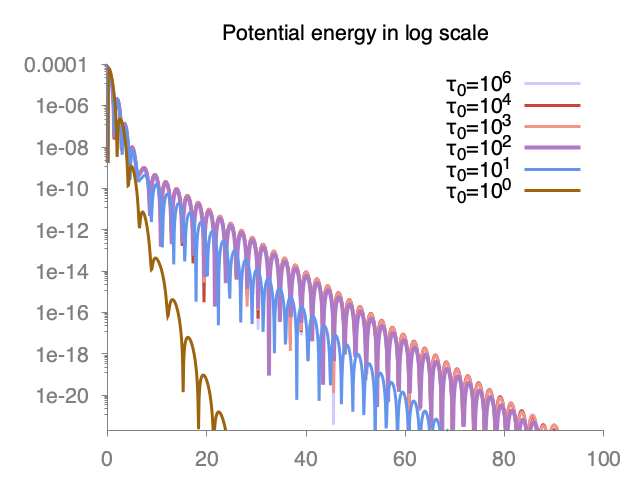} &
 \includegraphics[width=8.cm,angle=0]{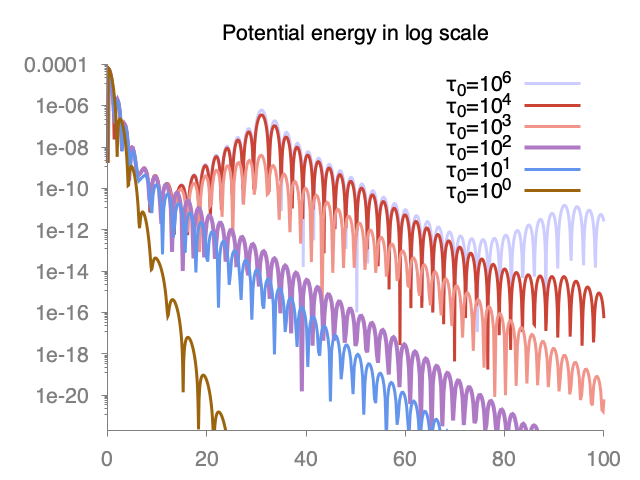}
 \\
 \includegraphics[width=8.cm,angle=0]{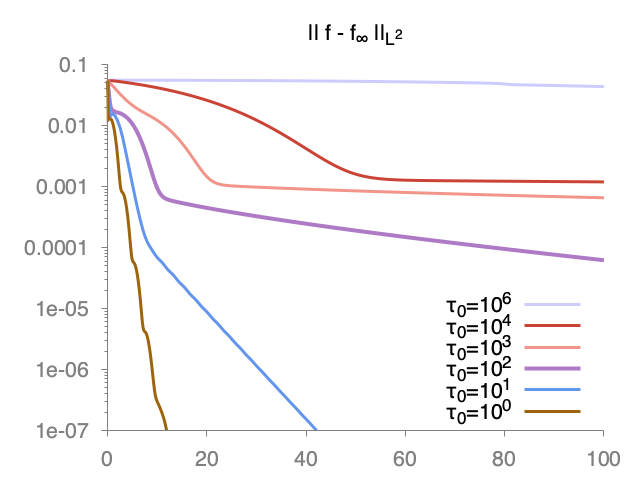} &
 \includegraphics[width=8.cm,angle=0]{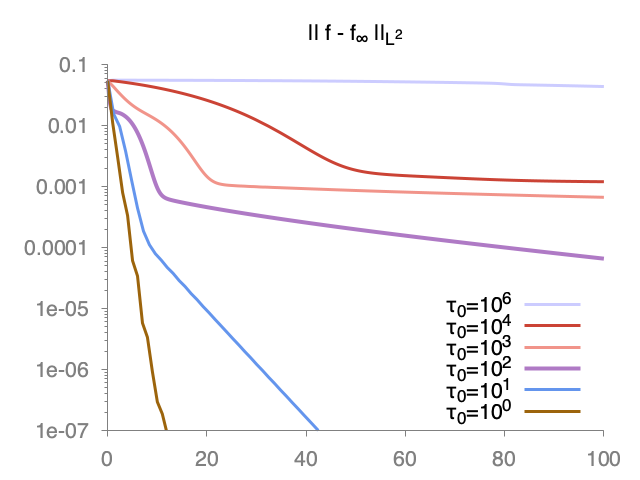}
  \\
  (a) & (b)
\end{tabular}
\caption{ {\it Plasma echo for $\tau_0=1,\ldots, 10^6$ (various regimes): time development of the potential energy (top) and $\| f-f_\infty \|_{L^2(f_\infty^{-1})}$ (bottom) in log scale for (a) the linearized Vlasov-Poisson-Fokker-Planck system and (b) the nonlinear Vlasov-Poisson-Fokker-Planck system.}}
\label{fig:1.4}
\end{figure}

\subsection{Two-stream}\label{sec:two:stream}

In this last experiment, we fix $\eps$ to $1$ and consider the equilibrium
$$
\phi_\infty(x) = 0.1\,\left( 1-\cos(k\,x)\right), \quad x \in (-L,L),
$$
with $k=\pi/L$ with $L=6$. The equilibrium is therefore given by 
\[
f_{\infty}(x,v)\,=\,c_0\,\exp\left(-\,\phi_{\infty}(x)\,-\,\frac{|v|^2}{2}\right)\,,
\]
where $c_0$ is renormalizing constant. Furthermore, we consider the initial distribution function as 
$$
f(t=0,x,v) \,=\, \f{1}{6\,\sqrt{2\pi}}\,\left(1+\delta\cos(k\,x)\right)\, (1+5v^2)
\, \exp\left(-\frac{v^2}{2}\right)\,, \quad (x,v) \in (-L,L)\times \R,
$$
where $\delta=10^{-2}$. These conditions can be viewed as a perturbation
of data leading to the well-known two-stream instability when
$\phi_\infty\equiv 0$. For this case, we fix the number of Hermite modes at $N_H=800$ and consider the solution $f$ to the nonlinear scheme \eqref{discrete:step1}-\eqref{discrete:step2}. The purpose of this experiment is to compare our results with the classical two-stream instability which is usually performed with an homogeneous background distribution and to study the influence of collisions on our results.\\
Our first comment is that unlike in the classical two-stream
instability \cite{ref:5}, it is not clear that the electric field
develops an exponential growth in this case. This may be observed on the left plot of Figure \ref{fig:3.1} considering the curves of the quantity $\| E -E_\infty \|_{L^2(\T)}$ in weak and intermediate collisional regimes $\tau_0=10^k$, with $2\leq k\leq6$, and also on the bottom charts of Figure \ref{fig:3.2} considering the blue curves which represent the time development of $\|\rho-\rho_{\infty}\|_{L^2(\rho_\infty^{-1})}$ when $\tau_0=10^k$, with $k=2\,,3$. \\
However, similarly to classical two-stream instabilities, we remark that in weakly collisional regimes ($\tau_0=10^k$, with $3\leq k\leq6$), the electric field is not damped over time since collisions are negligible on the timescale of our experiments (see Figure \ref{fig:3.1}). When collisions are extremely weak $k=4\,,6$, we even distinguish oscillations of the electric field (see Figure \ref{fig:3.1}). \\
This similarity with the classical two-stream instability may also be noticed at the level of kinetic distributions, as we may see on the left-hand side of Figure \ref{fig:3.3} where are represented snapshots of
$f$ at different times when $\tau_0=10^4$. Indeed, we witness the formation of a vortex persisting over time. \\
When collisions are intense enough, that is $\tau_0=10^k$, with $k\leq2$, we perceive exponential relaxation towards equilibrium at the level of the electric field (see left chart of Figure \ref{fig:3.1}), kinetic distribution (see right chart of Figure \ref{fig:3.1}), spatial density and higher Hermite modes (see Figure \ref{fig:3.2}). This relaxation may also be observed on Figure \ref{fig:3.3}, columns $(b)$ and $(c)$, where the vortex structure is affected by collisions and even canceled completely when $\tau_0=10^2$.\\
Our last comment on this experiment concerns the strongly collisional regime $\tau_0=10^k$ when $k=0\,,\,1$. A somehow surprising phenomena unfolds since new oscillations appear on all the quantities of interest: electric field (see left chart of Figure \ref{fig:3.1}), kinetic distribution (see right chart of Figure \ref{fig:3.1}), spatial density and higher Hermite modes (see Figure \ref{fig:3.2}, plots $(a)$ and $(b)$). We have already discerned oscillations in a similar setting \cite[Section $4.1$]{BF_09_22} where a non-constant stationary force field was applied in strongly collisional settings. However we deal here with a self induced force field whereas \cite{BF_09_22} focuses on the linear case. These oscillations seem robust enough to persist in the present situation.

\begin{figure}[htbp]
\begin{tabular}{cc}
\includegraphics[width=8.cm,angle=0]{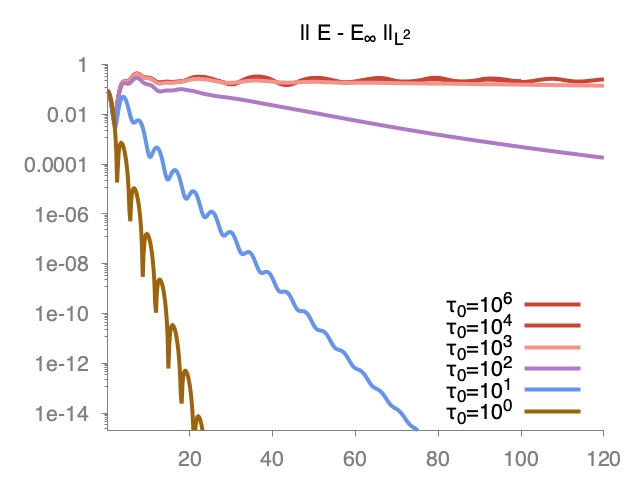}&
\includegraphics[width=8.cm,angle=0]{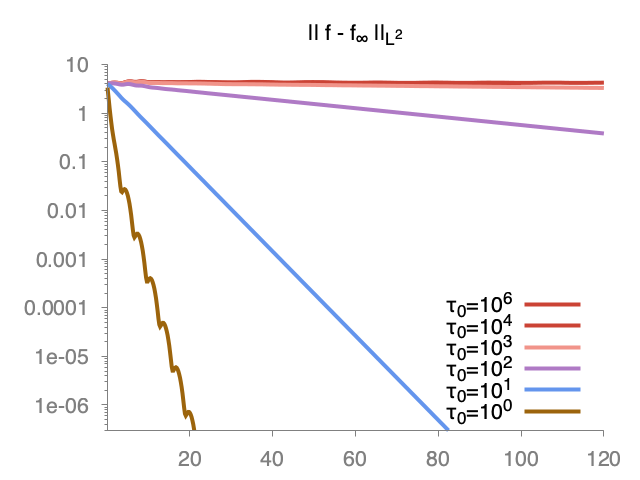}
\end{tabular}
\caption{ {\it Two-stream : time development of (a) $\| E -E_\infty \|_{L^2(\T)}$ (b) and $\| f -f_\infty \|_{L^2(f_\infty^{-1})}$ for various $\tau_0=10^2,\ldots, 10^6$ (in log scale).}}
\label{fig:3.1}
\end{figure}

\begin{figure}[htbp]
\begin{tabular}{cc}
\includegraphics[width=8.cm,angle=0]{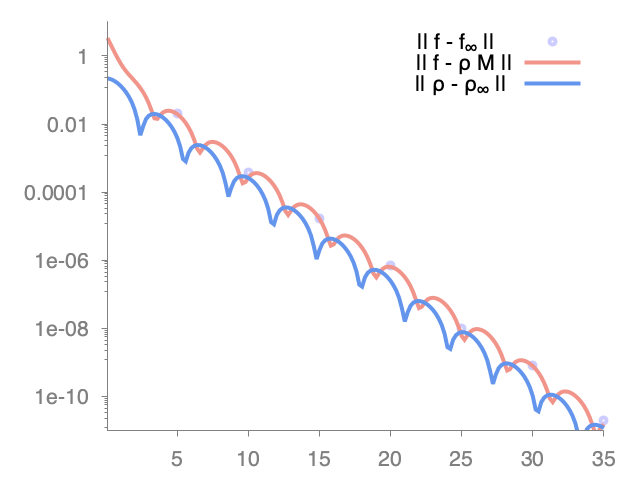}&
\includegraphics[width=8.cm,angle=0]{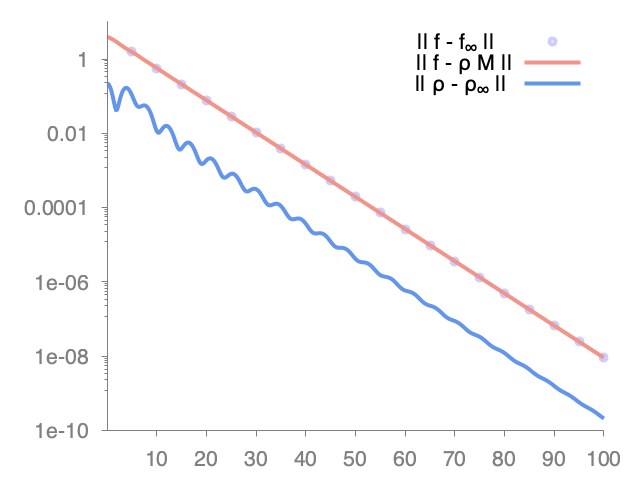}
\\ (a) $\tau_0=1$ & (b) $\tau_0=10$
\\
 \includegraphics[width=8.cm,angle=0]{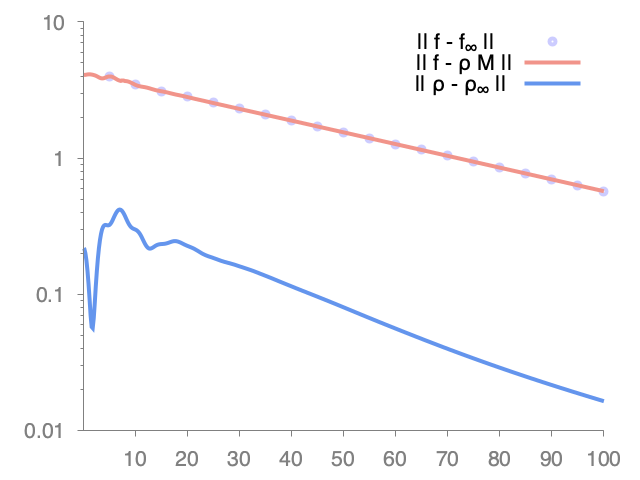}&
\includegraphics[width=8.cm,angle=0]{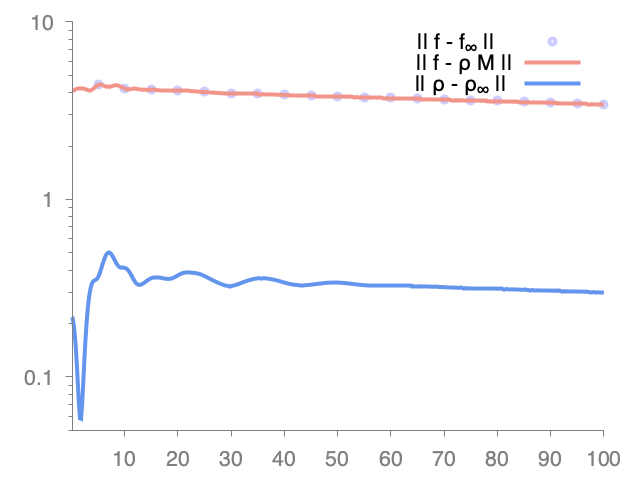}\\ 
(c) $\tau_0=100$ & (d) $\tau_0=1000$ 
\end{tabular}
\caption{ {\it Two-stream : time development of $\| f -f_\infty\|_{L^2(f_\infty^{-1})}$, $\| f -\rho\, \cM \|_{L^2(f_\infty^{-1})}$ and $\| \rho -\rho_\infty\|_{L^2(\rho_\infty^{-1})}$ for various $\tau_0$ (in log scale).}}
\label{fig:3.2}
\end{figure}

\begin{figure}[htbp]
\begin{tabular}{ccc}
\includegraphics[width=5.cm,angle=0]{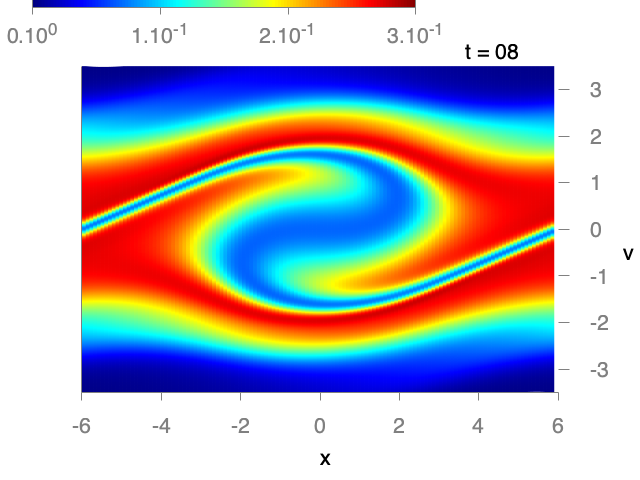} &
\includegraphics[width=5.cm,angle=0]{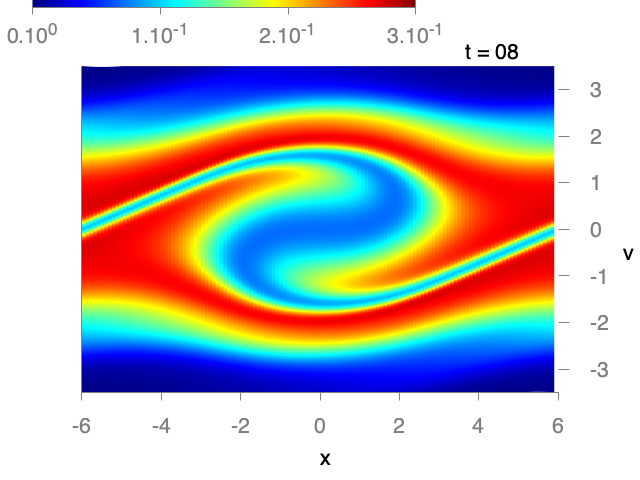} &
\includegraphics[width=5.cm,angle=0]{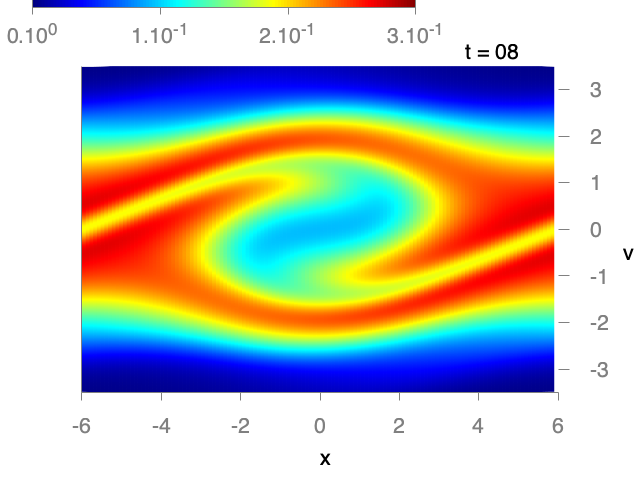}\\
\includegraphics[width=5.cm,angle=0]{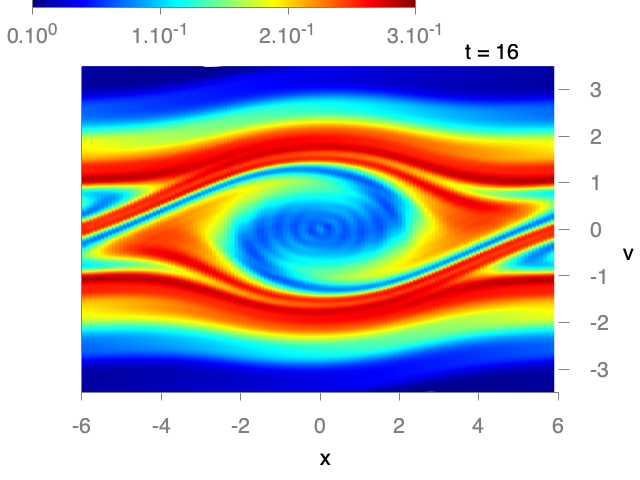} &
\includegraphics[width=5.cm,angle=0]{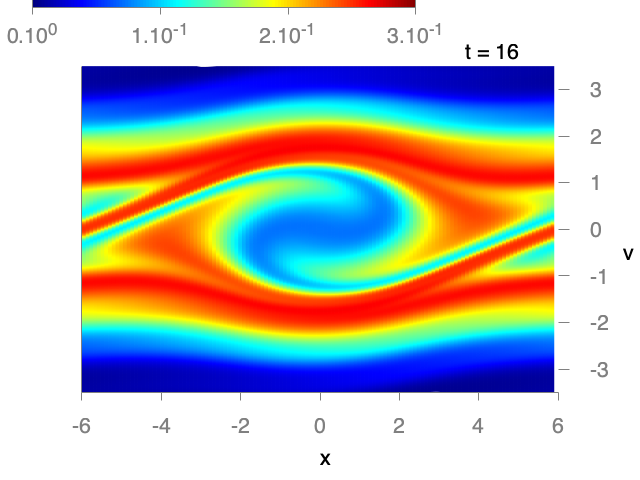} &
\includegraphics[width=5.cm,angle=0]{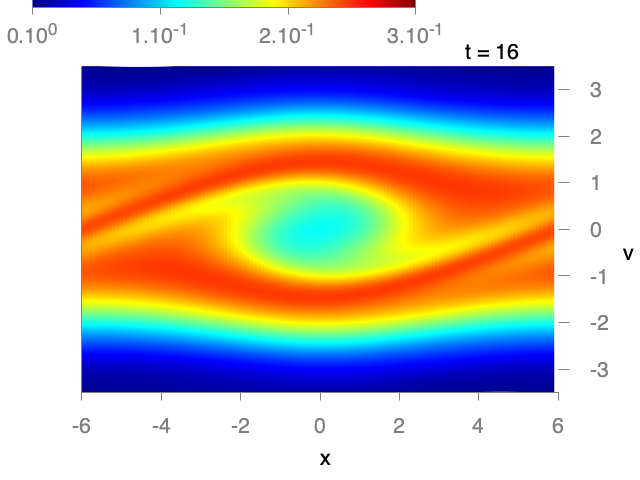}
\\
\includegraphics[width=5.cm,angle=0]{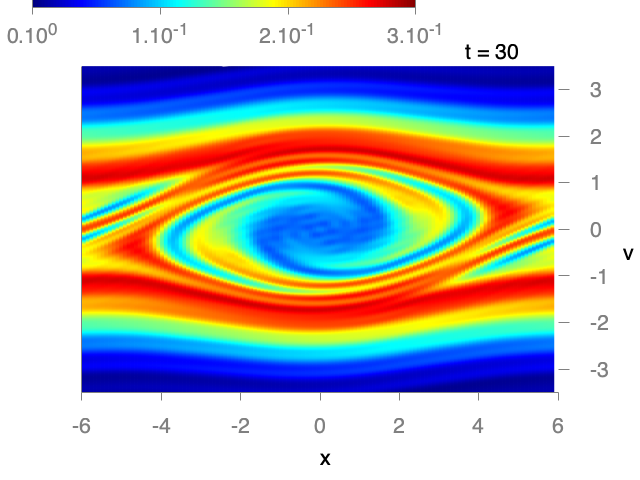} &
\includegraphics[width=5.cm,angle=0]{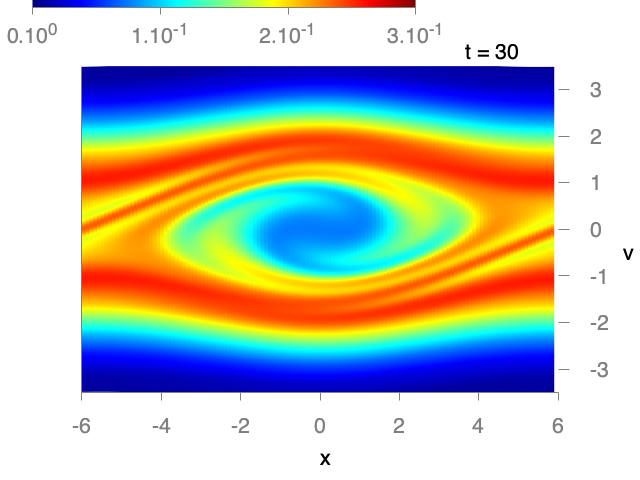} &
\includegraphics[width=5.cm,angle=0]{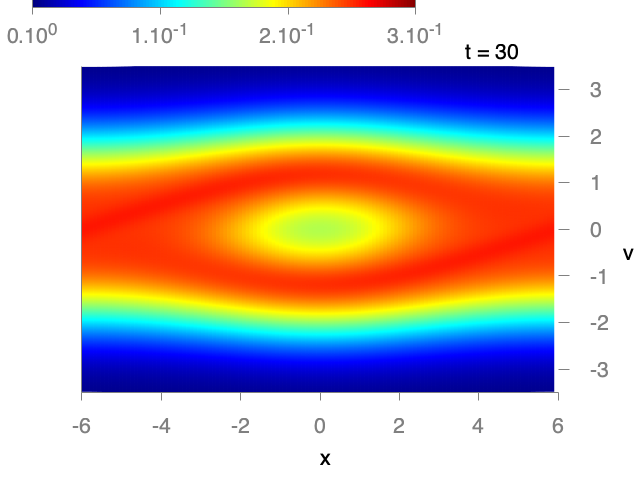}
\\
\includegraphics[width=5.cm,angle=0]{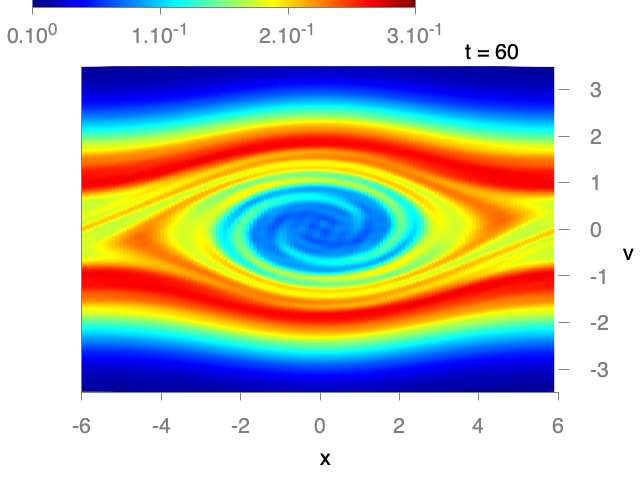} &
\includegraphics[width=5.cm,angle=0]{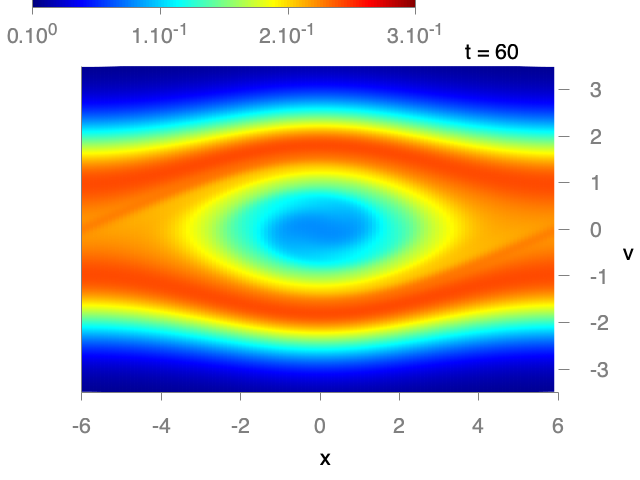} &
\includegraphics[width=5.cm,angle=0]{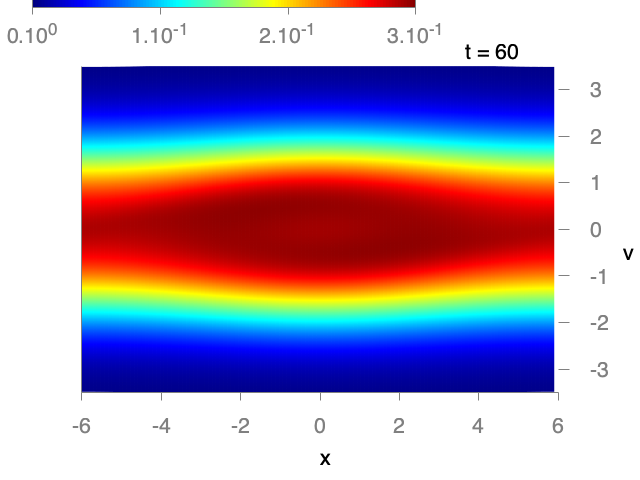}
  \\
  (a) $\tau_0=10^4$ & (b) $\tau_0=10^3$ & (c) $\tau_0=10^2$ 
\end{tabular}
\caption{ {\it Two-stream : snapshots of the distribution function $f$ at time $t=8$, $16$, $30$ and $60$ for various $\tau_0$.}}
\label{fig:3.3}
\end{figure}

\section{Conclusion and perspectives}
\label{sec:5}
\setcounter{equation}{0}
\setcounter{figure}{0}
\setcounter{table}{0}

In this work, we proposed a numerical scheme for the Vlasov-Poisson-Fokker-Planck model. On the one hand, we proved that our method is asymptotic preserving in the long time regime for the linearized model. To do so, we derived the exponential relaxation of the numerical solution towards its equilibrium with rates independent of scaling and discretization parameters. On the other hand, we tested the robustness of the method in various numerical experiments. These experiments show the accuracy of our method in both weakly collisional regime where small scales of the system are captured, allowing to reproduce filamentation, vortex formation as well as fine nonlinear phenomena such as plasma echoes but also in strong collisional regime, where we witness exponential trend to equilibrium, as predicted by our analysis of the linearized model.\\
Many interesting perspectives arise from this work. On the theoretical
view point, an important continuation consists in extending our
theoretical results, which apply for a linear coupling with the
Poisson equation, to the nonlinear scheme
\eqref{discrete:step1}-\eqref{discrete:step2} by proving its
asymptotic preserving properties and exponential trend towards
equilibrium of discrete solutions. This might be doable in a
perturbative setting by controlling the nonlinear contribution using
discrete Sobolev inequalities. Carrying such proof in higher
dimensions $d=2,3$ would be a great challenge and would probably
require new theoretical tools. Indeed, equivalent studies on the
continuous model in the literature rely on propagation of regularity
\cite{Herau_Thomann16,Hwang_Jang13,liu18, Herda_Rodrigues}. In our case it would require to propagate regularity at the discrete level. The groundwork towards such result has been laid in \cite{BF_09_22}, where we propagated discrete $H^1$ norms in the linear setting.\\
{Another important continuation of this work is to incorporate nonlinear collisions to the model. Let us first observe that in \cite{FN:22}, the Hermite spectral  method is applied to a
nonlinear Fokker-Planck operator  conserving mass, momentum and
energy. However, extending our analysis of the longtime regime at the discrete level to this case may require modifications and further investigations have to be done. $L^2$-hypocoercivity methods have been applied in the case of nonlinear BGK and linearized Boltzmann operators at the continuous level \cite{Herau06,Franz_Anton_Eric18}, however such analysis at the discrete level is not available in the literature in the framework of Hermite decomposition.}\\
Regarding simulations, the study of echoes also raises interesting perspectives. In \cite{Grenier_Toan_Rodnianski22} were constructed theoretical solutions to the Vlasov-Poisson equation which display infinite cascades of echoes and for which Landau damping is therefore not verified. Furthermore, sharp joint conditions on the collision frequency and the size of the initial data were obtained in order to ensure suppression of these echoes in \cite{Bedrossian17,Chaturvedi_Luk_Nguyen23}. Constructing such numerical solutions and illustrating the threshold obtained in these analysis would be of great interest. Another possible continuation would consist in finding a non-homogeneous background configuration where damping phenomena occur as in the homogeneous case, and then construct an experiment where nonlinear effect play the main role, even for small perturbation as for plasma echoes.

%%%%%%%%%%%%%%%%%%%%%%%%%%%%%%%%%%%
%
%%%%%%%%%%%%%%%%%%%%%%%%%%%%%%%%%%%
\section*{Acknowledgement}

Both authors are partially funded by the ANR Project Muffin
(ANR-19-CE46-0004). This work has been initiated during the semester
``Frontiers in kinetic theory: connecting microscopic to macroscopic
scales `` at the Isaac Newton Institute for Mathematical Sciences,
Cambridge, UK. 

\bibliographystyle{abbrv}
\bibliography{refer}
\end{document}